\documentclass[twosides, ammssymb, 11pt]{amsart}
\usepackage{xcolor}
\usepackage{color}
\usepackage{amssymb}
\usepackage{mathrsfs}
\usepackage{amsmath}
\usepackage{amsfonts}
\usepackage{amsfonts,amssymb}
\usepackage{amscd}
\usepackage{amsthm,amsmath,amssymb}
\usepackage{xypic}
\usepackage{latexsym}  
\usepackage{graphicx}

\def\C{\mathbb C}

\def\Q{\mathbb{C}_{q}[x,y,z]}
\def\X{X_{q}(A_{1})}

\newtheorem{thm}{Theorem}[section]
\newtheorem{lem}[thm]{Lemma}

\newtheorem{defn}[thm]{Definition}

\numberwithin{equation}{section}

\date{}
\begin{document}

\thispagestyle{empty}

\begin{center}
{\bf{\LARGE  Module algebra structures of nonstandard quantum group $X_{q}(A_{1})$  on  $\C_{q}[x,y,z]$} \footnotetext {
Supported by National Natural Science Foundation of China (Grant No.12201187) and  Natural Science Foundation of Henan Province(Grant No. 222300420156)

\small Corresponding author: Dong Su E-mail: sudong@haust.edu.cn }}

\bigbreak

\normalsize Dong Su 

{\footnotesize\small\sl School of Mathematics and Statistics,
Henan University of Science and Technology,
\footnotesize\sl Luoyang 471023, P. R.  China\\
}
\end{center}

\begin{quote}
{\noindent\small{\bf Abstract.}
In this paper,  the module algebra structures of $X_{q}(A_{1})$ on quantum polynomial algebra $\C_{q}[x,y,z]$ are investigated, and  a  complete classification of   $X_{q}(A_{1})$-module algebra structures  on $\C_{q}[x,y,z]$ is given.}

\end{quote}

\section*{Introduction}

The notion of Hopf algebra actions on algebras was introduced by Beattie \cite{BEATTIE1, BEATTIE2} in 1976. 
A duality theorem for Hopf module algebras was studied by Blattner and Montgomery \cite{BM} in 1985. 
It generalized the corresponding theorem of group actions. 
Moreover, the actions of Hopf algebras \cite{MONTGOMERY} and their generalizations (see, e.g.,\cite{DVZ}) play an
important role in quantum group theory \cite{KASSEL, KS} and  in its various applications in physics \cite{CW}.
Duplij and  Sinel'shchikov used a general form of the automorphism of  quantum plane to give the notion of weight for $U_{q}(sl_{2})$-actions considered here, 
and completely classified  quantum group $U_{q}(sl_{2})$-module algebra structures on the quantum plane \cite{DS1,DS2},
that the results are  much richer and consists of $8$ non-isomorphic cases.
Moreover, In \cite{DHL} the authors used the method of weights \cite{DS1,DS2} to classified some actions in terms of action matrices, the module algebra structures of quantum group $U_{q}(sl(m+1))$ on the coordinate algebra of quantum vector spaces were studied,
and the concrete actions of quantum group $U_{q}(sl_{2})$ on $\Q$ were researched (also can see \cite{ZML}).
More relevant research can be found at \cite{CWWZ, HU}.

The  non-standard  quantum groups were studied in \cite{Ge}, 
Ge et al. obtained new solutions of Yang-Baxter equations,
and gave including the twisted extensions quantum group structures  related to these new solutions explicitly.
In \cite{Agha} one class of non-standard quantum deformation corresponding to simple Lie algebra $sl_{n}$ was given, 
which is denoted by $X_q(A_{n-1})$.
For each vertex $i (i=1,\cdots,n)$  of the Dynkin diagram,
the parameter $q_i$  is equal to $q$  or $-q^{-1}$,
if $q_i=q$ for all $i$, then $X_q(A_{n-1})$ is  just to $U_q(sl_n)$.
However, if $q_i\neq q_{i+1}$ for some $1\leq i\leq n$, 
it has the relations $E_{i}^{2}=F_{i}^{2}=0$ in $X_q(A_{n-1})$. 
Such a $X_q(A_{n-1})$ is different to $U_q(sl_n)$.
Jing et al. \cite{N} derived a non-standard  quantum group by employing the FRT constructive method,
and classify all finite dimensional irreducible representations of this non-standard  quantum group.
Cheng and Yang \cite{CY} considered the structures and representations of weak Hopf algebras $\mathfrak{w}X_q(A_{1})$, which is corresponding to non-standard quantum group $X_q(A_{1})$.
We \cite{SY1} researched the representations of a class of  small nonstandard quantum groups $\overline{X}_{q}(A_1)$,
over which the isomorphism classes of all indecomposable modules are classified, and the decomposition formulas
of the tensor product of arbitrary indecomposable modules and simple (or
projective) modules are established. The projective class rings and
Grothendieck rings of $\overline{X}_{q}(A_1)$ are also characterized.
However, the research on module algebra of non-standard quantum groups has not yet yielded any results.
Consequently, based on the research results of module algebra of quantum groups, 
we consider here the module algebra of the nonstandard quantum group $X_{q}(A_{1})$  on the quantum polynomial algebra $\Q$.
and a complete list of  $X_{q}(A_{1})$-module algebra structures  on  $\Q$ is produced and the isomorphism classes of these structures are described.

This paper is organized as follows. 
In Section 1,  we introduce some necessary notations and the concepts.
In Section 2, when $t=0$, we discuss the module algebra structures of $\X$ on the polynomial algebra $\Q$ using the method of weights \cite{DS1,DS2}.
We study the concrete actions of  $\X$ on  $\Q$ and characterize all module algebra structures of $\X$ on  $\Q$.
In Section 3, we study the module algebra structures of $\X$ on  $\Q$ with $t\neq0$.
In the same way of section 2, We study the concrete actions of  $\X$ on  $\Q$ and characterize all module algebra structures of $\X$ on  $\Q$.

\section{Preliminaries}\label{sect-1}
Throughout, we work over the complex field $\mathbb{C}$ unless otherwise stated.
All algebras, Hopf algebras and modules are defined over $\mathbb{C}$;
all maps are $\mathbb{C}$-linear.

Let $(H,m,\eta,\Delta,\varepsilon,S)$ be a Hopf algebra, here $\Delta$ and $\varepsilon$ are the comultiplication and counit of $H$, respectively. Let $A$ be a unital algebra with unit $\bf{1}$. We  will also use the Sweedler notation
$\Delta(h)=\sum\limits_{i}h'_{i}\otimes h''_{i}$ \cite{SW}.

\begin{defn}\label{defn1-1}
By a structure of $H$-module algebra on $A$, we mean a homomorphism $\pi: H\rightarrow \mathrm{End}_{\C}A$ such that:
\begin{enumerate}
  \item for all $h\in H, a,b\in A$, $\pi(h)(ab)=\sum\limits_{i}\pi(h'_{i})(a)\cdot\pi(h''_{i})(b); $
  \item for all $h\in H,  \pi(h)(1)=\varepsilon(h)\bf{1}$.
\end{enumerate}
\end{defn}

The structures $\pi_{1}, \pi_{2}$ are said to be isomorphic if there exists an automorphism $\Psi$ of the algebra $A$ such that $\Psi\pi_{1}(h)\Psi^{-1}=\pi_{2}(h)$ for all $h\in H.$

We assume that $q\in \C^{*}=\C\setminus\{0\}$ is not a root of the unit ($q^{n}\neq 1$ for all non-zero integers $n$).
A class of non-standard quantum algebra $X_q(A_1)$ was studied by Jing etc. \cite{N}.
By definition the algebra $X_q(A_1)$  is a unital associative $\C$-algebra generated by $E, F, K_i, K_{i}^{-1} (i=1,2)$ subject to the relations:
\begin{eqnarray}\label{eqn1-1}
&&K_{i}K_{i}^{-1}=K_{i}^{-1}K_{i}=1,~~K_{1}K_{2}=K_{2}K_{1},\\\label{eqn1-2} 
&&K_{1}E=q^{-1}EK_{1},~~ \\\label{eqn1-3}
&&K_{1}F=qFK_{1},\\\label{eqn1-4}
&&K_{2}E=-q^{-1}EK_{2},~~ \\\label{eqn1-5}
&&K_{2}F=-qFK_{2},\\\label{eqn1-6}
&&EF-FE=\frac{K_{2}K_{1}^{-1}-K_{2}^{-1}K_{1}}{q-q^{-1}},\\ \label{eqn1-7}
&&E^{2}=F^{2}=0.
\end{eqnarray}
The algebra $\X$ is also a Hopf algebra, the comultiplication $\Delta$, counit $\varepsilon$ and antipode $S$ are given as the following
\begin{eqnarray}\label{eqn1-8}
&&\Delta(K_{i})=K_{i}\otimes K_{i}, \\ \label{eqn1-9}
&&\Delta(E)=E\otimes1+K_{2}K_{1}^{-1}\otimes E,\\ \label{eqn1-10}
&&\Delta(F)=1\otimes F+F\otimes  K_{2}^{-1}K_{1}, 
\end{eqnarray}
\begin{eqnarray*}
\varepsilon(K_{i})=1,&\varepsilon(E)=0, &\varepsilon(F)=0,\\
S(K_{i})=K_{i}^{-1}, &S(E)=-K_{1}K_{2}^{-1}E, &S(F)=-FK_{2}K_{1}^{-1}.
\end{eqnarray*}

We consider the quantum polynomial algebra $\C_{q}[x,y,z]$  is a unital algebra, generated by  generators $x,y,z$, and  satisfying the relations
\begin{eqnarray}\label{eqn1-11}
&&yx=qxy,\\\label{eqn1-12}
&&zy=qyz,\\\label{eqn1-13}
&&zx=qxz.
\end{eqnarray}

Denote by $\Q_{s}$ the $s$-th homogeneous component of $\Q$, which is a
linear span of the monomials $x^{m_{1}}y^{m_{2}}z^{m_{3}}$ with $m_{1}+m_{2}+m_{3}=s$.
Also, given a polynomial
$p\in\Q$, denote by $p_{s}$ the $s$-th homogeneous component of $p$, that is the
projection of $p$ onto $\Q_{s}$ parallel to the direct sum of all other homogeneous
components of $\Q$.

By \cite{ALCH,ART1,ART2,ARWI}, one has
a description of automorphisms of the algebra $\Q$, as follows.
Let $\Psi$ be an  automorphism of $\Q$, then there exist nonzero constants  $\alpha,\beta,\gamma\in \mathbb{C}^{*}$ and $t\in \mathbb{C}$, such that
$$\Psi:x\rightarrow \alpha x,~~~y\rightarrow \beta y+txz,~~~z\rightarrow  \gamma z.$$
All such automorphisms form the automorphism group of $\Q$ denoted by $\mathrm{Aut}(\Q)$, one can get $\mathrm{Aut}(\Q)\cong \mathbb{C}\rtimes\left(\mathbb{C}^{*}\right)^{3}$.
In the following sections, we will explore the classificaion of $X_{q}(A_{1})$-module algebra structures on $\C_{q}[x,y,z]$.

\section{When $t=0$, classification of  $X_{q}(A_{1})$-module algebra structures  on $\Q$   }\label{sect-2}

In this section, our aim is to  describe the  $X_{q}(A_{1})$-module algebra structures  on  $\Q$,  with $t=0$, ie. the automorphism of $\Q$ as follows
$$\Psi(x)=\alpha x,~~~\Psi(y)=\beta y,~~~\Psi(z)=\gamma z,~~~(\alpha,\beta,\gamma\in \mathbb{C}^{*}),$$
and $\mathrm{Aut}(\Q)\cong \left(\mathbb{C}^{*}\right)^{3}$,
here $K_{1},K_{2}\in \mathrm{Aut}(\Q)$
\subsection{Properties of  $X_{q}(A_{1})$-module algebras  on $\Q$  }\label{sect-2.1}

By the definition of module algebra, it is easy to see that any action of $\X$ on $\Q$ is determined by the following $4\times 3$ matrix with entries from $\Q$:
\begin{eqnarray}\label{eqn2-1}
M\overset{definition}=\left(
    \begin{array}{cccccc}
      K_{1}(x) & K_{1}(y) & K_{1}(z)\\
      K_{2}(x) & K_{2}(y) & K_{2}(z)\\
      E(x) & E(y) & E(z)\\
      F(x) & F(y) & F(z)
    \end{array}
  \right),
\end{eqnarray}
which is called the full action matrix.
Given a $\X$-module algebra structure on $\Q$, obviously, the action of $K_{1}$(or $K_{2}$) is determined by an automorphism of $\Q$,
in other words, the actions  of $K_{1}$ and $K_{2}$ are determined by a matrix $M_{K_{1}K_{2}}$ as follows
\begin{eqnarray}\label{eqn2-02}
\begin{array}{lllllll}
M_{K_{1}K_{2}}\overset{definition}=&\left(
    \begin{array}{cccccc}
    K_{1}(x) & K_{1}(y) & K_{1}(z)\\
      K_{2}(x) & K_{2}(y) & K_{2}(z)
    \end{array}
  \right)&~&~&~\\
&=\left(
    \begin{array}{cccccc}
   \alpha_{1}(x) & \beta_{1}(y) & \gamma_{1}(z)\\
    \alpha_{2}(x) & \beta_{2}(y) & \gamma_{2}(z)\\
    \end{array}
  \right),
\end{array}
\end{eqnarray}
where $\alpha_{i},\beta_{i}\in \mathbb{C}^{*}$ for $i\in\{1,2\}$.
It is easy to see that every monomial $x^{m_{1}}y^{m_{2}}z^{m_{3}}\in\Q$ is an eigenvector of $K_{1}$(or $K_{2}$),
and the associated eigenvalue $\alpha_{1}^{m_{1}}\beta_{1}^{m_{2}}\gamma_{1}^{m_{3}}$ (or $\alpha_{2}^{m_{1}}\beta_{2}^{m_{2}}\gamma_{2}^{m_{3}}$) is  called the $K_{1}$-weight (or $K_{2}$-weight) of this monomial,
which will be written as
$$wt_{K_{1}}(x^{m_{1}}y^{m_{2}}z^{m_{3}})=\alpha_{1}^{m_{1}}\beta_{1}^{m_{2}}\gamma_{1}^{m_{3}},$$
$$wt_{K_{2}}(x^{m_{1}}y^{m_{2}}z^{m_{3}})=\alpha_{2}^{m_{1}}\beta_{2}^{m_{2}}\gamma_{2}^{m_{3}}.$$

We will also need another matrix  $M_{EF}$ as follows
\begin{eqnarray}
M_{EF}\overset{definition}=\left(
    \begin{array}{cccccc}
      E(x) & E(y) & E(z)\\
      F(x) & F(y) & F(z)
    \end{array}
  \right).
\end{eqnarray}\label{eqn2-3}

Obviously, all entries of $M$ are weight vectors for $K_{1}$ and $K_{2}$,  then
\begin{eqnarray}
\begin{array}{lllllll}
wt_{K_{i}}\left(M\right)&\overset{definition}=\left(
    \begin{array}{cccccc}
    wt_{K_{i}}(K_{1}(x)) &  wt_{K_{i}}(K_{1}(y)) & wt_{K_{i}}(K_{1}(z))\\
      wt_{K_{i}}(K_{2}(x)) & wt_{K_{i}}(K_{2}(y)) & wt_{K_{i}}(K_{2}(z))\\
      wt_{K_{i}}(E(x)) & wt_{K_{i}}(E(y)) & wt_{K_{i}}(E(z))\\
      wt_{K_{i}}(F(x)) & wt_{K_{i}}(F(y)) & wt_{K_{i}}(F(z))
    \end{array}
  \right)&~&~&~\\
~&\bowtie \left(
    \begin{array}{cccccc}
    wt_{K_{i}}(x) &  wt_{K_{i}}(y)& wt_{K_{i}}(z)\\
      wt_{K_{i}}(x) & wt_{K_{i}}(y) & wt_{K_{i}}(z)\\
     (-1)^{i-1}q^{-1}wt_{K_{i}}(x) &  (-1)^{i-1}q^{-1}wt_{K_{i}}(y) & (-1)^{i-1}q^{-1}wt_{K_{i}}(z)\\
     (-1)^{i-1}qwt_{K_{i}}(x) & (-1)^{i-1}qwt_{K_{i}}(y) & (-1)^{i-1}qwt_{K_{i}}(z)
    \end{array}
  \right)&~&~&~\\
~&= \left(
    \begin{array}{cccccc}
    \alpha_{i} &  \beta_{i} & \gamma_{i}\\
    \alpha_{i} &  \beta_{i} & \gamma_{i}\\
     (-1)^{i-1}q^{-1}\alpha_{i} &  (-1)^{i-1}q^{-1}\beta_{i} & (-1)^{i-1}q^{-1}\gamma_{i}\\
      (-1)^{i-1}q\alpha_{i} &  (-1)^{i-1}q\beta_{i} & (-1)^{i-1}q\gamma_{i}
    \end{array}
  \right),
\end{array}
\end{eqnarray}\label{eqn2-4}
where the relation $A=\left(a_{st}\right)\bowtie B=\left(b_{st}\right)$ means that for every pair of indices $s,t$ such that
both $a_{st}$ and $b_{st}$ are nonzero, one has $a_{st}=b_{st}$.

We denote by $\left(M\right)_{j}$ the $j$-th homogeneous component of $M$,
whose elements are just the $j$-th homogeneous components of the corresponding entries of $M$. Set
\begin{eqnarray}\left(M\right)_{0}=\left(
    \begin{array}{cccccc}
      0 &  0 & 0\\
      0 &  0 & 0\\
      a_{0} &  b_{0} & c_{0}\\
     a'_{0} &  b'_{0} & c'_{0}
    \end{array}
  \right)_{0},
\end{eqnarray}\label{eqn2-5}
where, $a_{0}, b_{0},c_{0},a'_{0},  b'_{0}, c'_{0}\in \C$.
Then,  we obtain
\begin{eqnarray}\label{eqn2-07}
\begin{array}{lllllll}
 wt_{K_{1}}\left(\left(M_{EF}\right)_{0}\right)&\bowtie\left(
    \begin{array}{cccccc}
     q^{-1}\alpha_{1} &  q^{-1}\beta_{1} & q^{-1}\gamma_{1}\\
      q\alpha_{1} &  q\beta_{1} & q\gamma_{1}
    \end{array}
  \right)_{0}\bowtie \left(
    \begin{array}{cccccc}
   1 &  1 & 1\\
   1 &  1 & 1
    \end{array}
  \right)_{0},
\end{array}
\end{eqnarray}
\begin{eqnarray}\label{eqn2-08}
\begin{array}{lllllll}
 wt_{K_{2}}\left(\left(M_{EF}\right)_{0}\right)&\bowtie\left(
    \begin{array}{cccccc}
     -q^{-1}\alpha_{2} &  -q^{-1}\beta_{2} & -q^{-1}\gamma_{2}\\
      -q\alpha_{2} &  -q\beta_{2} & -q\gamma_{2}
    \end{array}
  \right)_{0}\bowtie \left(
    \begin{array}{cccccc}
   1 &  1& 1\\
   1 &  1 & 1
    \end{array}
  \right)_{0}.
\end{array}
\end{eqnarray}

According to $q$ is not a root of the unit and relations (\ref{eqn2-07})-(\ref{eqn2-08}), it means that each column of $M_{EF}$ should contain at least one $0$.

An application of $E$ and $F$ to the relations (\ref{eqn1-11})-(\ref{eqn1-13}) by using equation (\ref{eqn2-02}), one has
\begin{eqnarray}\label{eqn2-8}
E(y)x+\beta_{1}^{-1}\beta_{2}yE(x)=qE(x)y+q\alpha_{1}^{-1}\alpha_{2}xE(y), \\ \label{eqn2-9}
E(z)y+\gamma_{1}^{-1}\gamma_{2}zE(y)=qE(y)z+q\beta_{1}^{-1}\beta_{2}yE(z), \\ \label{eqn2-10}
E(z)x+\gamma_{1}^{-1}\gamma_{2}zE(x)=qE(x)z+q\alpha_{1}^{-1}\alpha_{2}xE(z), \\ \label{eqn2-11}
yF(x)+\alpha_{2}^{-1}\alpha_{1}F(y)x=qxF(y)+q\beta_{2}^{-1}\beta_{1}F(x)y, \\\label{eqn2-12}
zF(y)+\beta_{2}^{-1}\beta_{1}F(z)y=qyF(z)+q\gamma_{2}^{-1}\gamma_{1}F(y)z, \\\label{eqn2-13}
zF(x)+\alpha_{2}^{-1}\alpha_{1}F(z)x=qxF(z)+q\gamma_{2}^{-1}\gamma_{1}F(x)z. 
\end{eqnarray}

After projecting equations (\ref{eqn2-8})-(\ref{eqn2-13}) to $\Q_{1}$, we obtain
\begin{eqnarray*}
b_{0}\left(1-q\alpha_{1}^{-1}\alpha_{2}\right)x+a_{0}\left(\beta_{1}^{-1}\beta_{2}-q\right)y=0,\\
c_{0}\left(1-q\beta_{1}^{-1}\beta_{2}\right)y+b_{0}\left(\gamma_{1}^{-1}\gamma_{2}-q\right)z=0,\\
c_{0}\left(1-q\alpha_{1}^{-1}\alpha_{2}\right)x+a_{0}\left(\gamma_{1}^{-1}\gamma_{2}-q\right)z=0,\\
a'_{0}\left(1-q\beta_{1}\beta_{2}^{-1}\right)y+b'_{0}\left(\alpha_{1}\alpha_{2}^{-1}-q\right)x=0,\\
b'_{0}\left(1-q\gamma_{1}\gamma_{2}^{-1}\right)z+c'_{0}\left(\beta_{1}\beta_{2}^{-1}-q\right)y=0,\\
a'_{0}\left(1-q\gamma_{1}\gamma_{2}^{-1}\right)z+c'_{0}\left(\alpha_{1}\alpha_{2}^{-1}-q\right)x=0,
\end{eqnarray*}
which certainly implies
\begin{eqnarray*}
b_{0}\left(1-q\alpha_{1}^{-1}\alpha_{2}\right)=a_{0}\left(\beta_{1}^{-1}\beta_{2}-q\right)=
c_{0}\left(1-q\beta_{1}^{-1}\beta_{2}\right)=b_{0}\left(\gamma_{1}^{-1}\gamma_{2}-q\right)=0,\\
c_{0}\left(1-q\alpha_{1}^{-1}\alpha_{2}\right)=a_{0}\left(\gamma_{1}^{-1}\gamma_{2}-q\right)=
a'_{0}\left(1-q\beta_{1}\beta_{2}^{-1}\right)=b'_{0}\left(\alpha_{1}\alpha_{2}^{-1}-q\right)=0,\\
b'_{0}\left(1-q\gamma_{1}\gamma_{2}^{-1}\right)=c'_{0}\left(\beta_{1}\beta_{2}^{-1}-q\right)=
a'_{0}\left(1-q\gamma_{1}\gamma_{2}^{-1}\right)=c'_{0}\left(\alpha_{1}\alpha_{2}^{-1}-q\right)=0.
\end{eqnarray*}
We will determine the weight constants $\alpha_{i}$, $\beta_{i}$ and $\gamma_{i} (i=1,2)$  as follows:
\begin{eqnarray}
\begin{array}{lllllll}
a_{0}\neq 0&\Rightarrow  &\alpha_{1}=q,    &\alpha_{2}=-q, &\beta_{1}^{-1}\beta_{2}=q, &\gamma_{1}^{-1}\gamma_{2}=q;\\\label{eqn2-11}
b_{0}\neq 0&\Rightarrow  &\beta_{1}=q,     &\beta_{2}=-q, &\alpha_{1}^{-1}\alpha_{2}=q^{-1}, &\gamma_{1}^{-1}\gamma_{2}=q;\\\label{eqn2-11}
c_{0}\neq 0&\Rightarrow  &\gamma_{1}=q,    &\gamma_{2}=-q, &\beta_{1}^{-1}\beta_{2}=q^{-1}, &\alpha_{1}^{-1}\alpha_{2}=q^{-1};\\\label{eqn2-11}
a'_{0}\neq 0&\Rightarrow &\alpha_{1}=q^{-1}, &\alpha_{2}=-q^{-1}, &\beta_{1}\beta_{2}^{-1}=q^{-1}, &\gamma_{1}\gamma_{2}^{-1}=q^{-1};\\
b'_{0}\neq 0&\Rightarrow &\beta_{1}=q^{-1}, &\beta_{2}=-q^{-1}, &\alpha_{1}\alpha_{2}^{-1}=q, & \gamma_{1}\gamma_{2}^{-1}=q^{-1};\\\label{eqn2-11}
c'_{0}\neq 0&\Rightarrow &\gamma_{1}=q^{-1}, &\gamma_{2}=-q^{-1}, &\beta_{1}\beta_{2}^{-1}=q, &\alpha_{1}\alpha_{2}^{-1}=q.\label{eqn2-11}
\end{array}
\end{eqnarray}

Because $q$ is not a root of the unit, $q\neq \pm1$.
Therefore at least one of $a_{0}$, $b_0$, $c_0$ and $a'_{0}$, $b'_0$, $c'_0$ is not zero.
In summary, we have obtained the following results for the $0$-st homogeneous component $\left(M_{EF}\right)_{0}$ of  $M_{EF}$.
\begin{lem}\label{lem2-1}
There are $7$ cases for the $0$-st homogeneous component $\left(M_{EF}\right)_{0}$ of  $M_{EF}$, as follows:
\begin{eqnarray}
&&\left(
    \begin{array}{cccccc}
      a_{0} &  0 & 0\\
      0 &  0 & 0
    \end{array}
  \right)_{0}\Rightarrow\alpha_{1}=q,    \alpha_{2}=-q, \beta_{1}^{-1}\beta_{2}=q, \gamma_{1}^{-1}\gamma_{2}=q;\\
&&\left(
    \begin{array}{cccccc}
      0 &  b_{0}& 0\\
      0 &  0& 0
    \end{array}
  \right)_{0}\Rightarrow\beta_{1}=q,     \beta_{2}=-q, \alpha_{1}^{-1}\alpha_{2}=q^{-1}, \gamma_{1}^{-1}\gamma_{2}=q;\\
&&\left(
    \begin{array}{cccccc}
      0 &  0 & c_{0}\\
      0 &  0 & 0
    \end{array}
  \right)_{0}\Rightarrow\gamma_{1}=q,    \gamma_{2}=-q, \beta_{1}^{-1}\beta_{2}=q^{-1}, \alpha_{1}^{-1}\alpha_{2}=q^{-1};\\
&&\left(
    \begin{array}{cccccc}
      0 &  0 & 0\\
      a'_{0} &  0 & 0
    \end{array}
  \right)_{0}\Rightarrow\alpha_{1}=q^{-1}, \alpha_{2}=-q^{-1}, \beta_{1}\beta_{2}^{-1}=q^{-1}, \gamma_{1}\gamma_{2}^{-1}=q^{-1};\\
&&\left(
    \begin{array}{cccccc}
      0 &  0 & 0\\
      0 &  b'_{0} & 0
    \end{array}
  \right)_{0}\Rightarrow\beta_{1}=q^{-1}, \beta_{2}=-q^{-1}, \alpha_{1}\alpha_{2}^{-1}=q, \gamma_{1}\gamma_{2}^{-1}=q^{-1};\\
&&\left(
    \begin{array}{cccccc}
    0 &  0 & 0\\
      0 &  0 & c'_{0}
    \end{array}
  \right)_{0},\Rightarrow\gamma_{1}=q^{-1}, \gamma_{2}=-q^{-1}, \beta_{1}\beta_{2}^{-1}=q, \alpha_{1}\alpha_{2}^{-1}=q;
\end{eqnarray}
\begin{eqnarray}
&& \left(
    \begin{array}{cccccc}
    0 &  0 & 0\\
      0 &  0 & 0
    \end{array}
  \right)_{0}\mathrm{it~ does ~not~ determine~ the~ weight~ constants~ at ~all.} 
\end{eqnarray}
\end{lem}

Next, for the $1$-st homogeneous component, due to $q$ is not a root of the unit, one has
\begin{eqnarray*}
&&wt_{K_{1}}(E(x)) = q^{-1}\alpha_{1}=q^{-1}wt_{K_{1}}(x)\neq wt_{K_{1}}(x),\\
&&wt_{K_{2}}(E(x)) =-q^{-1}\alpha_{2}=-q^{-1}wt_{K_{2}}(x)\neq wt_{K_{2}}(x),
\end{eqnarray*}
which implies
$$(E(x))_{1}=a_{1}y+a_{2}z,$$
for some  $ a_{1},a_{2}\in \C$,
and in a similar way we have
{\small$$\left(M_{EF}\right)_{1}=\left(\begin{array}{cccccc}
a_{1}y+a_{2}z   &  b_{1}x+b_{2}z    &c_{1}x+c_{2}y\\
a'_{1}y+a'_{2}z &  b'_{1}x+b'_{2}z  &c'_{1}x+c'_{2}y
    \end{array}
  \right)_{1}$$}
where $b_{1},b_{2},c_{1},c_{2},a'_{1},a'_{2},b'_{1},b'_{2},c'_{1},c'_{2}\in \C$.
Therefore
\begin{eqnarray}
\begin{array}{lllllll}
wt_{K_{i}}\left((M_{EF})_{1}\right)&\bowtie \left(
    \begin{array}{cccccc}
   (-1)^{i-1}q^{-1}\alpha_{i} &  (-1)^{i-1}q^{-1}\beta_{i} & (-1)^{i-1}q^{-1}\gamma_{i}\\
      (-1)^{i-1}q\alpha_{i} &  (-1)^{i-1}q\beta_{i} & (-1)^{i-1}q\gamma_{i}
    \end{array}
  \right)&~&~&~\\
~&\bowtie \left(
    \begin{array}{cccccc}
     \beta_{i}~ \mathrm{or}~  \gamma_{i}&   \alpha_{i}~ \mathrm{or}~\gamma_{i} & \alpha_{i}~ \mathrm{or}~  \beta_{i}\\
     \beta_{i}~ \mathrm{or}~  \gamma_{i}&   \alpha_{i}~ \mathrm{or}~\gamma_{i} & \alpha_{i}~ \mathrm{or}~  \beta_{i}
    \end{array}
  \right),
\end{array}
\end{eqnarray}\label{eqn2-4}

Now project (\ref{eqn2-8})-(\ref{eqn2-13}) to $\Q_{2}$ to obtain
\begin{eqnarray*}
&&b_{1}\left(1-q\alpha_{1}^{-1}\alpha_{2}\right)x^{2}+b_{2}\left(1-\alpha_{1}^{-1}\alpha_{2}\right)zx
+a_{1}\left(\beta_{1}^{-1}\beta_{2}-q\right)y^{2}+a_{2}\left(\beta_{1}^{-1}\beta_{2}-q^{2}\right)yz=0,\\
&&c_{1}\left(1-q^{2}\beta_{1}^{-1}\beta_{2}\right)yx+c_{2}\left(1-q\beta_{1}^{-1}\beta_{2}\right)y^{2}
+b_{2}\left(\gamma_{1}^{-1}\gamma_{2}-q\right)z^{2}+b_{1}q\left(\gamma_{1}^{-1}\gamma_{2}-1\right)xz=0,\\
&&c_{1}\left(1-q\alpha_{1}^{-1}\alpha_{2}\right)x^{2}+qc_{2}\left(1-\alpha_{1}^{-1}\alpha_{2}\right)xy
+qa_{1}\left(\gamma_{1}^{-1}\gamma_{2}-1\right)yz+a_{2}\left(\gamma_{1}^{-1}\gamma_{2}-q\right)z^{2}=0,\\
&&a'_{1}\left(1-q\beta_{2}^{-1}\beta_{1}\right)y^{2}+a'_{2}\left(1-q^{2}\beta_{2}^{-1}\beta_{1}\right)yz
+b'_{1}\left(\alpha_{2}^{-1}\alpha_{1}-q\right)x^{2}+qb'_{2}\left(\alpha_{2}^{-1}\alpha_{1}-1\right)xz=0,\\
&&qb'_{1}\left(1-\gamma_{2}^{-1}\gamma_{1}\right)xz+b'_{2}\left(1-q\gamma_{2}^{-1}\gamma_{1}\right)z^{2}
+c'_{1}\left(\beta_{2}^{-1}\beta_{1}-q^{2}\right)xy+c'_{2}\left(\beta_{2}^{-1}\beta_{1}-q\right)y^{2}=0,\\
&&qa'_{1}\left(1-\gamma_{2}^{-1}\gamma_{1}\right)yz+a'_{2}\left(1-q\gamma_{2}^{-1}\gamma_{1}\right)z^{2}
+c'_{1}\left(\alpha_{2}^{-1}\alpha_{1}-q\right)x^{2}+qc'_{2}\left(\alpha_{2}^{-1}\alpha_{1}-1\right)xy=0,
\end{eqnarray*}
which certainly implies
\begin{eqnarray*}
&&b_{1}\left(1-q\alpha_{1}^{-1}\alpha_{2}\right)=b_{2}\left(1-\alpha_{1}^{-1}\alpha_{2}\right)
=a_{1}\left(\beta_{1}^{-1}\beta_{2}-q\right)=a_{2}\left(\beta_{1}^{-1}\beta_{2}-q^{2}\right)=0\\
&&c_{1}\left(1-q^{2}\beta_{1}^{-1}\beta_{2}\right)=c_{2}\left(1-q\beta_{1}^{-1}\beta_{2}\right)
=b_{2}\left(\gamma_{1}^{-1}\gamma_{2}-q\right)=b_{1}q\left(\gamma_{1}^{-1}\gamma_{2}-1\right)=0\\
&&c_{1}\left(1-q\alpha_{1}^{-1}\alpha_{2}\right)=qc_{2}\left(1-\alpha_{1}^{-1}\alpha_{2}\right)
=qa_{1}\left(\gamma_{1}^{-1}\gamma_{2}-1\right)=a_{2}q\left(\gamma_{1}^{-1}\gamma_{2}-q\right)=0\\
&&a'_{1}\left(1-q\beta_{1}\beta_{2}^{-1}\right)=a'_{2}\left(1-q^{2}\beta_{1}\beta_{2}^{-1}\right)
=b'_{1}\left(\alpha_{1}\alpha_{2}^{-1}-q\right)=qb'_{2}\left(\alpha_{1}\alpha_{2}^{-1}-1\right)=0,\\
&&qb'_{1}\left(1-\gamma_{1}\gamma_{2}^{-1}\right)=b'_{2}\left(1-q\gamma_{1}\gamma_{2}^{-1}\right)
=c'_{1}\left(\beta_{1}\beta_{2}^{-1}-q^{2}\right)=c'_{2}\left(\beta_{1}\beta_{2}^{-1}-q\right)=0,\\
&&qa'_{1}\left(1-\gamma_{1}\gamma_{2}^{-1}\right)=a'_{2}\left(1-q\gamma_{1}\gamma_{2}^{-1}\right)
=c'_{1}\left(\alpha_{1}\alpha_{2}^{-1}-q\right)=qc'_{2}\left(\alpha_{1}\alpha_{2}^{-1}-1\right)=0.
\end{eqnarray*}

As a consequence, we have
\begin{eqnarray}
\begin{array}{llllllllll}
&a_{1}\neq 0&\Rightarrow& \beta_{2}\alpha_{1}^{-1}=q,       &\gamma_{2}\gamma_{1}^{-1}=1,\\
&a_{2}\neq 0&\Rightarrow& \beta_{2}\alpha_{1}^{-1}=q^{2},   &\gamma_{2}\gamma_{1}^{-1}=q,\\
&b_{1}\neq 0&\Rightarrow& \alpha_{2}\alpha_{1}^{-1}=q^{-1}, &\gamma_{2}\gamma_{1}^{-1}=1,\\
&b_{2}\neq 0&\Rightarrow& \alpha_{2}\alpha_{1}^{-1}=1,      &\gamma_{2}\gamma_{1}^{-1}=q,\\
&c_{1}\neq 0&\Rightarrow& \beta_{2}\alpha_{1}^{-1}=q^{-2}   &\alpha_{2}\alpha_{1}^{-1}=q^{-1}, \\
&c_{2}\neq 0&\Rightarrow& \beta_{2}\alpha_{1}^{-1}=q^{-1}   &\alpha_{2}\alpha_{1}^{-1}=1,
\end{array}
\end{eqnarray}
\begin{eqnarray}
\begin{array}{llllllllll}
&a'_{1}\neq 0&\Rightarrow& \beta_{2}\alpha_{1}^{-1}=q,       &\gamma_{2}\gamma_{1}^{-1}=1,\\
&a'_{2}\neq 0&\Rightarrow& \beta_{2}\alpha_{1}^{-1}=q^{2},   &\gamma_{2}\gamma_{1}^{-1}=q,\\
&b'_{1}\neq 0&\Rightarrow& \alpha_{2}\alpha_{1}^{-1}=q^{-1}, &\gamma_{2}\gamma_{1}^{-1}=1,\\
&b'_{2}\neq 0&\Rightarrow& \alpha_{2}\alpha_{1}^{-1}=1,      &\gamma_{2}\gamma_{1}^{-1}=q,\\
&c'_{1}\neq 0&\Rightarrow& \beta_{2}\alpha_{1}^{-1}=q^{-2}   &\alpha_{2}\alpha_{1}^{-1}=q^{-1}, \\
&c'_{2}\neq 0&\Rightarrow& \beta_{2}\alpha_{1}^{-1}=q^{-1}   &\alpha_{2}\alpha_{1}^{-1}=1.
\end{array}
\end{eqnarray}

From the above discussion,  for the $1$-st homogeneous component $\left(M_{EF}\right)_{1}$ of  $M_{EF}$, we have following lemma.
\begin{lem}\label{lem2-2}
There are $13$ cases for the $1$-st homogeneous component $\left(M_{EF}\right)_{1}$ of  $M_{EF}$, as follows:
\begin{eqnarray}
&&\left(
    \begin{array}{ccccccccc}
      a_{1}y & 0&0\\
      0      & 0&0
    \end{array}
  \right)_{1} \Rightarrow \beta_{1}=q^{-1}\alpha_{1}, \beta_{2}=-q^{-1}\alpha_{2},\beta_{1}^{-1}\beta_{2}=q, \gamma_{1}^{-1}\gamma_{2}=1;\\
&&\left(
    \begin{array}{ccccccccc}
      a_{2}z & 0&0\\
      0      & 0&0
    \end{array}
  \right)_{1} \Rightarrow \gamma_{1}=q^{-1}\alpha_{1}, \gamma_{2}=-q^{-1}\alpha_{2},\beta_{1}^{-1}\beta_{2}=q^{2}, \gamma_{1}^{-1}\gamma_{2}=q;\\
&&\left(
    \begin{array}{ccccccccc}
      0 &b_{1}x&0\\
      0      & 0&0
    \end{array}
  \right)_{1} \Rightarrow \beta_{1}=q\alpha_{1}, \beta_{2}=-q\alpha_{2},\alpha_{1}^{-1}\alpha_{2}=q^{-1}, \gamma_{1}^{-1}\gamma_{2}=1;\\
&&\left(
    \begin{array}{ccccccccc}
      0 &b_{2}z&0\\
      0      & 0&0
    \end{array}
  \right)_{1}\Rightarrow \gamma_{1}=q^{-1}\beta_{1}, \gamma_{2}=-q^{-1}\beta_{2},\alpha_{1}^{-1}\alpha_{2}=1, \gamma_{1}^{-1}\gamma_{2}=q;\\
&&\left(
    \begin{array}{ccccccccc}
      0 &0&c_{1}x\\
      0 & 0&0
    \end{array}
  \right)_{1} \Rightarrow \alpha_{1}=q^{-1}\gamma_{1}, \alpha_{2}=-q^{-1}\gamma_{2},\alpha_{1}^{-1}\alpha_{2}=q^{-1}, \beta_{1}^{-1}\beta_{2}=q^{-2};\\
&&\left(
    \begin{array}{ccccccccc}
      0 &0&c_{2}y\\
      0      & 0&0
    \end{array}
  \right)_{1}\Rightarrow \beta_{1}=q^{-1}\gamma_{1}, \beta_{2}=-q^{-1}\gamma_{2},\alpha_{1}^{-1}\alpha_{2}=1, \beta_{1}^{-1}\beta_{2}=q^{-1};\\
&&\left(
    \begin{array}{ccccccccc}
       0& 0&0\\
      a'_{1}y      & 0&0
    \end{array}
  \right)_{1} \Rightarrow \beta_{1}=q\alpha_{1}, \beta_{2}=-q\alpha_{2},\beta_{1}^{-1}\beta_{2}=q, \gamma_{1}^{-1}\gamma_{2}=1;\\
&&\left(
    \begin{array}{ccccccccc}
      0 & 0&0\\
      a'_{2}z & 0&0
    \end{array}
  \right)_{1} \Rightarrow \gamma_{1}=q\alpha_{1}, \gamma_{2}=-q\alpha_{2},\beta_{1}^{-1}\beta_{2}=q^{2}, \gamma_{1}^{-1}\gamma_{2}=q;\\
&&\left(
    \begin{array}{ccccccccc}
      0 &0&0\\
      0 &b'_{1}x&0
    \end{array}
  \right)_{1}\Rightarrow \beta_{1}=q^{-1}\alpha_{1}, \beta_{2}=-q^{-1}\alpha_{2},\alpha_{1}^{-1}\alpha_{2}=q^{-1}, \gamma_{1}^{-1}\gamma_{2}=1;
\end{eqnarray}
\begin{eqnarray}
&&\left(
    \begin{array}{ccccccccc}
      0 &0&0\\
      0& b'_{2}z&0
    \end{array}
  \right)_{1} \Rightarrow \gamma_{1}=q\beta_{1}, \gamma_{2}=-q\beta_{2},\alpha_{1}^{-1}\alpha_{2}=1, \gamma_{1}^{-1}\gamma_{2}=q;\\
&&\left(
    \begin{array}{ccccccccc}
      0 &0&0\\
      0 & 0&c'_{1}x
    \end{array}
  \right)_{1} \Rightarrow \alpha_{1}=q\gamma_{1}, \alpha_{2}=-q\gamma_{2},\alpha_{1}^{-1}\alpha_{2}=q^{-1}, \beta_{1}^{-1}\beta_{2}=q^{-2};\\
&&\left(
    \begin{array}{ccccccccc}
      0 &0&0\\
      0 &0&c'_{2}y
    \end{array}
  \right)_{1} \Rightarrow \beta_{1}=q\gamma_{1}, \beta_{2}=-q\gamma_{2},\alpha_{1}^{-1}\alpha_{2}=1, \beta_{1}^{-1}\beta_{2}=q^{-1};\\
&&\left(
    \begin{array}{ccccccccc}
      0 &0&0\\
      0 &0&0
    \end{array}
  \right)_{1} \mathrm{it~ does~ not~ determine~ the~ weight~ constants ~at~ all.}
\end{eqnarray}
\end{lem}

\subsection{The structures of  $X_{q}(A_{1})$-module algebra  on $\Q$  }\label{sect-2.2}

In this subsection,  our aim is to  describe the concrete  $X_{q}(A_{1})$-module algebra structures  on $\Q$, 
where $K_{1},K_{2}\in \mathrm{Aut}(\Q)\cong \left(\mathbb{C}^{*}\right)^{3}$.

By Lemma \ref{lem2-1} and \ref{lem2-2}, and $q$ is not a root of the unit, it follows that if both the $0$-th homogeneous component and the $1$-th homogeneous component of $M_{EF}$ are nonzero, it is easy to see that these series are empty. So, we need to consider following possibilities.

\begin{lem}\label{lem2-3} If  the $0$-th homogeneous component of $M_{EF}$ is zero and the $1$-st homogeneous component of $M_{EF}$ is nonzero, then these  series are empty. 
\end{lem}
\begin{proof} Now, we  show that
 {\small$\left[\left(
    \begin{array}{ccccccccc}
      0 & 0&0\\
      0     &0&0
    \end{array}
  \right)_{0},\left(
    \begin{array}{ccccccccc}
      a_1y & 0&0\\
      0 & 0&0
    \end{array}
  \right)_{1}\right]$}
-series is empty.
If we suppose the contrary, then it follows from
$$EF-FE=\frac{K_{2}K_{1}^{-1}-K_{2}^{-1}K_{1}}{q-q^{-1}}$$
that within this series, one can have
$$\frac{K_{2}K_{1}^{-1}-K_{2}^{-1}K_{1}}{q-q^{-1}}(x)=
\frac{\alpha_{2}\alpha_{1}^{-1}-\alpha_{2}^{-1}\alpha_{1}}{q-q^{-1}}x.$$
By $a_{1}\neq0$, one can get $\beta_{1}=q^{-1}\alpha_{1}, \beta_{2}=-q^{-1}\alpha_{2},$
$\beta_{2}\beta_{1}^{-1}=q$, and $\gamma_{2}\gamma_{1}^{-1}=1$, hence $\alpha_{2}\alpha_{1}^{-1}=-q$, and
$$\frac{K_{2}K_{1}^{-1}-K_{2}^{-1}K_{1}}{q-q^{-1}}(x)=-x.$$
On the other hand, projecting $(EF-FE)(x)$ to $\Q$ we obtain
$$(EF-FE)(x)=E(F(x))-F(E(x))=E(0)-F(a_{1}y)=0,$$
however, $0\neq -x$.
We have obtained contradictions and proved our claims.

In a similar way, one can prove that all other series with the $0$-th homogeneous component of $M_{EF}$ is zero and the $1$-th homogeneous component of $M_{EF}$ is nonzero  are empty. 
\end{proof}

\begin{lem}\label{lem2-4} If  the $0$-th homogeneous component of $M_{EF}$ is nonzero and the $1$-st homogeneous component of $M_{EF}$ is zero, then these  series are empty. 
\end{lem}
\begin{proof}We only  show that
{\small$\left[\left(
    \begin{array}{ccccccccc}
      a_{0} & 0&0\\
      0     &0&0
    \end{array}
  \right)_{0},\left(
    \begin{array}{ccccccccc}
      0 & 0&0\\
      0 & 0&0
    \end{array}
  \right)_{1}\right]$}
-series is empty. in a similar way, one can prove that all other series are empty.

Consider this series, we obtain that
\begin{eqnarray*}
\begin{array}{llllllllll}
&a_0\neq0&\Rightarrow& \alpha_{1}=q, & \alpha_{2}=-q,&\beta_{2}\beta_{1}^{-1}=q, &\gamma_{2}\gamma_{1}^{-1}=q.
\end{array}
\end{eqnarray*}
and  suppose that it is not empty. We set
 \begin{eqnarray*}
\begin{array}{lllllll}
&K_{1}(x)=qx,\; \quad\quad \quad K_{2}(x)=-qx,\; &\\
&K_{1}(y)=\beta_{1}y,\; \quad\quad \quad K_{2}(y)=\beta_{2}y,\; &\\
&K_{1}(z)=\gamma_{1}z,\; \quad\quad \quad K_{2}(z)=\gamma_{2}z,\; &\\
&E(x)=a_{0}+\sum\limits_{m_{1}+n_{1}+l_{1}\geq2}\rho^{1}_{m_{1}n_{1}l_{1}}x^{m_{1}}y^{n_{1}}z^{l_{1}}\; &\mathrm{for}~~{m_{1},n_{1}, l_{1}}\in \mathbb{N},\\
&E(y)=\sum\limits_{m_{2}+n_{2}+l_{2}\geq2}\rho^{2}_{m_{2}n_{2}l_{2}}x^{m_{2}}y^{n_{2}}z^{l_{2}}\;
&\mathrm{for}~~{m_{2},n_{2},l_{2}}\in \mathbb{N},\\
&E(z)=\sum\limits_{m_{3}+n_{3}+l_{3}\geq2}\rho^{3}_{m_{3}n_{3}l_{3}}x^{m_{3}}y^{n_{3}}z^{l_{3}}\;
&\mathrm{for}~~{m_{3},n_{3},l_{3}}\in \mathbb{N},\\
&F(x)=\sum\limits_{m_{4}+n_{4}+l_{4}\geq2}\rho^{4}_{m_{4}n_{4}l_{4}}x^{m_{4}}y^{n_{4}}z^{l_{4}}\;
&\mathrm{for}~~{m_{4},n_{4},l_{4}}\in \mathbb{N},\\
&F(y)=\sum\limits_{m_{5}+n_{5}+l_{5}\geq2}\rho^{5}_{m_{5}n_{5}l_{5}}x^{m_{5}}y^{n_{5}}z^{l_{5}}\;
&\mathrm{for}~~{m_{5},n_{5},l_{5}}\in \mathbb{N},\\
&F(z)=\sum\limits_{m_{6}+n_{6}+l_{6}\geq2}\rho^{6}_{m_{6}n_{4}l_{6}}x^{m_{6}}y^{n_{6}}z^{l_{6}}\;
&\mathrm{for}~~{m_{6},n_{6},l_{6}}\in \mathbb{N},
\end{array}
\end{eqnarray*}
where  $\beta_{1},\beta_{2},\gamma_{1},\gamma_{2}\in \mathbb{C}^{\ast}$, 
and $\rho^{i}_{m_{i}n_{i}l_{i}}\in \C, i=1,2,3,4,5,6$.
We have
{\small\begin{align*}
(K_{1}E-q^{-1}EK_{1})(x)=&K_{1}(E(x))-q^{-1}E(K_{1}(x))\\
=&K_{1}(a_{0}+\sum\limits_{m_{1}+n_{1}+l_{1}\geq2}\rho^{1}_{m_{1}n_{1}l_{1}}x^{m_{1}}y^{n_{1}}z^{l_{1}})-q^{-1}qE(x)\\
=&a_{0}+\sum\limits_{m_{1}+n_{1}+l_{1}\geq2}\rho^{1}_{m_{1}n_{1}l_{1}}\alpha_{1}^{m_{1}}\beta_{1}^{n_{1}}\gamma_{1}^{l_{1}} x^{m_{1}}y^{n_{1}}z^{l_{1}}
-E(x)\\
=&\sum\limits_{m_{1}+n_{1}+l_{1}\geq2}\rho^{1}_{m_{1}n_{1}l_{1}}(q^{m_{1}}\beta_{1}^{n_{1}}\gamma_{1}^{l_{1}}-1) x^{m_{1}}y^{n_{1}}z^{l_{1}}=0,
\end{align*}}
then for all $m_{1},n_{1},l_{1}\in \mathbb{N}$ with $m_{1}+n_{1}+l_{1}\geq2$, one has
$\rho^{1}_{m_{1}n_{1}l_{1}}=0~~\mathrm{or} ~~q^{m_{1}}\beta_{1}^{n_{1}}\gamma_{1}^{l_{1}}=1.$
And
{\small\begin{align*}
(K_{2}E+q^{-1}EK_{2})(x)=&K_{2}(E(x))+q^{-1}E(K_{2}(x))\\
=&K_{2}(a_{0}+\sum\limits_{m_{1}+n_{1}+l_{1}\geq2}\rho^{1}_{m_{1}n_{1}l_{1}}x^{m_{1}}y^{n_{1}}z^{l_{1}})-q^{-1}qE(x)\\
=&a_{0}+\sum\limits_{m_{1}+n_{1}+l_{1}\geq2}\rho^{1}_{m_{1}n_{1}l_{1}}\alpha_{2}^{m_{1}}\beta_{2}^{n_{1}}\gamma_{2}^{l_{1}} x^{m_{1}}y^{n_{1}}z^{l_{1}}
-E(x)\\
=&\sum\limits_{m_{1}+n_{1}+l_{1}\geq2}\rho^{1}_{m_{1}n_{1}l_{1}}((-q)^{m_{1}}\beta_{2}^{n_{1}}\gamma_{2}^{l_{1}}-1) x^{m_{1}}y^{n_{1}}z^{l_{1}}=0,
\end{align*}}
then for all $m_{1},n_{1},l_{1}\in \mathbb{N}$ with $m_{1}+n_{1}+l_{1}\geq2$, one has
$\rho^{1}_{m_{1}n_{1}l_{1}}=0~~\mathrm{or} ~~(-q)^{m_{1}}\beta_{2}^{n_{1}}\gamma_{2}^{l_{1}}=1.$
If some $\rho^{1}_{m_{1}n_{1}l_{1}}\neq0$ meet the conditions, i.e
\begin{eqnarray*}
\left\{\begin{array}{ll}
q^{m_{1}}\beta_{1}^{n_{1}}\gamma_{1}^{l_{1}}=1,   \\
(-q)^{m_{1}}\beta_{2}^{n_{1}}\gamma_{2}^{l_{1}}=1,
\end{array}
\right.
\end{eqnarray*}
one can get
$(-1)^{m_{1}}q^{(n_{1}+l_{1})}=1,$
this contradicts with  $q$ is not a unit root.
Therefore, for all $m_{1},n_{1},l_{1}\in \mathbb{N}$ with $m_{1}+n_{1}+l_{1}\geq2$, we have
$E(x)=a_{0}.$
By discussing $E(y)$, $E(z)$, $F(x)$, $F(y)$ and $F(z)$ using methods similar to $E(x)$, we can obtain that
$$E(y)=0,~~E(z)=0,$$
\begin{eqnarray*}F(x)=
\left\{\begin{array}{ll}
0   \\
\rho^{4}_{200}x^{2}
\end{array}
\right.,\quad\quad
F(y)=
\left\{\begin{array}{ll}
0  \\
\rho^{5}_{110}xy
\end{array}
\right.,\quad\quad
F(z)=
\left\{\begin{array}{ll}
0   \\
\rho^{6}_{101}xz
\end{array}
\right..
\end{eqnarray*}

From $EF-FE=\frac{K_{2}K_{1}^{-1}-K_{2}^{-1}K_{1}}{q-q^{-1}}$, we have
\begin{eqnarray*}
  &&\frac{K_{2}K_{1}^{-1}-K_{2}^{-1}K_{1}}{q-q^{-1}}(y)=\frac{\beta_{2}\beta_{1}^{-1}-\beta_{2}^{-1}\beta_{1}}{q-q^{-1}y}
  =\frac{q-q^{-1}}{q-q^{-1}y}=y, \\
  &&\frac{K_{2}K_{1}^{-1}-K_{2}^{-1}K_{1}}{q-q^{-1}}(z)=\frac{\gamma_{2}\gamma_{1}^{-1}-\gamma_{2}^{-1}\gamma_{1}}{q-q^{-1}y}
  =\frac{q-q^{-1}}{q-q^{-1}z}=z.
\end{eqnarray*}
\begin{itemize}
  \item If $F(y)=0$, then $(EF-FE)(y)=EF(y)-FE(y)=0\neq y;$
  \item if $F(y)=\rho^{5}_{110}xy$, then $(EF-FE)(y)=\rho^{5}_{110}E(xy)=\rho^{5}_{110}a_{0}y=y$, hence $\rho^{5}_{110}=a_{0}^{-1}$ and $F(y)=a_{0}^{-1}xy.$
  \item if $F(z)=0$, then $(EF-FE)(z)=EF(z)-FE(z)=0\neq z;$
  \item if $F(z)=\rho^{6}_{101}xz$, then $(EF-FE)(z)=\rho^{6}_{101}E(xz)=\rho^{6}_{110}a_{0}z=z$, hence $\rho^{6}_{101}=a_{0}^{-1}$ and $F(z)=a_{0}^{-1}xz.$
\end{itemize}

By $E^{2}=F^{2}=0$, one has $F^{2}(y)=F(a_{0}^{-1}xy)=a_{0}^{-1}(xF(y)+q^{-1}F(x)y),$
\begin{itemize}
  \item if $F(x)=0$, then $F^{2}(y)=a_{0}^{-2}x^{2}y\neq 0;$
  \item if $F(x)=\rho^{4}_{200}x^{2}$, then $F^{2}(y)=a_{0}^{-1}(a_{0}^{-1}x^{2}y+\rho^{4}_{200}q^{-1}x^{2}y)=0,$ hence $\rho^{4}_{200}=-qa_{0}^{-1}$ and $F(x)=-qa_{0}^{-1}x^{2}.$
\end{itemize}

According to $yx=qxy$, then
\begin{align*}
F(yx-qxy)&=yF(x)-qxF(y)+\alpha_{2}^{-1}\alpha_{1}F(y)x-\beta^{-1}_{2}\beta_{1}qF(x)y\\
&=-a_{0}^{-1}(q^{3}+q)x^{2}y\neq 0.
\end{align*}
In summary, this series is empty.
\end{proof}

Next we turn to "nonempty" series, it only has a kind of "nonempty" series.

\begin{thm}\label{thm2-5} The {\small$\left[\left(
    \begin{array}{ccccccccc}
     0 & 0&0\\
      0 & 0&0
    \end{array}
  \right)_{0},\left(
    \begin{array}{ccccccccc}
      0 & 0&0\\
      0 & 0&0
    \end{array}
  \right)_{1}\right]$}
-series has $X_{q}(A_1)-$module algebra structures  on  $\Q$ given by
\begin{eqnarray}\label{eqn2-38}
&&K_1(x)=\lambda_1x, ~~K_2(x)=\pm\lambda_1x,\\\label{eqn2-39}
&&K_1(y)=\lambda_2y, ~~K_2(y)=\pm\lambda_2y,\\\label{eqn2-40}
&&K_1(z)=\lambda_3z, ~~K_2(z)=\pm\lambda_3z,\\\label{eqn2-41}
&&E(x)=F(x)=E(y)=F(y)=E(z)=F(z)=0,
\end{eqnarray}
where $\lambda_1, \lambda_2, \lambda_3\in \C^{*}$, they are pairwise nonisomorphic.
\end{thm}
\begin{proof}
It is easy to check that (\ref{eqn2-38})-(\ref{eqn2-41}) determine a well-defined $X_{q}(A_1)$-action consistent with the multiplication in $X_{q}(A_1)$ and in  $\Q$, as
well as with comultiplication in $X_{q}(A_1)$. 
Prove that there are no other $X_{q}(A_1)$-actions here. 
Note that an application  of (\ref{eqn1-6}) to $x, y$ or $z$ has zero
projection to $\Q_{1}$, ie. $(EF-FE)(x)=(EF-FE)(y)=(EF-FE)(z)=0$,  
because in this series $E$ and $F$ send any monomial to a sum
of the monomials of higher degree. 
Therefore,
\begin{eqnarray*}
  &&\frac{K_{2}K_{1}^{-1}-K_{2}^{-1}K_{1}}{q-q^{-1}}(x)=\frac{\alpha_{2}\alpha_{1}^{-1}
  -\alpha_{1}\alpha_{2}^{-1}}{q-q^{-1}}x=0, \\
  &&\frac{K_{2}K_{1}^{-1}-K_{2}^{-1}K_{1}}{q-q^{-1}}(y)=\frac{\beta_{2}\beta_{1}^{-1}
  -\beta_{1}\beta_{2}^{-1}}{q-q^{-1}}y=0,\\
  &&\frac{K_{2}K_{1}^{-1}-K_{2}^{-1}K_{1}}{q-q^{-1}}(z)=\frac{\gamma_{2}\gamma_{1}^{-1}
  -\gamma_{2}^{-1}\gamma_{1}}{q-q^{-1}}z=0, 
\end{eqnarray*}
and hence 
\begin{eqnarray*}
\alpha_{2}\alpha_{1}^{-1}-\alpha_{2}^{-1}\alpha_{2}=
\beta_{2}\beta_{1}^{-1}-\beta_{2}^{-1}\beta_{1}=
\gamma_{2}\gamma_{1}^{-1}-\gamma_{2}^{-1}\gamma_{2}=0,
\end{eqnarray*}
which leads to $\alpha_{1}^{2}=\alpha_{2}^{2}$, $\beta_{1}^{2}=\beta_{2}^{2}$ and $\gamma_{1}^{2}=\gamma_{2}^{2}$,
let $\alpha_{1}=\lambda_{1}$, $\beta_{1}=\lambda_{2}$ and $\gamma_{1}=\lambda_{3}$,
we have   $\alpha_{2}=\pm\lambda_{1}$, $\beta_{2}=\pm\lambda_{2}$ and $\gamma_{2}=\pm\lambda_{3}$.
To prove (\ref{eqn2-41}),
note that if $E(x)\neq 0$ or $F(y)\neq 0$, then they are  a sum of the monomials that their degrees are greater than 1.  It is similar to the proof of  Lemma \ref{lem2-4},
we get that this is impossible, because they can not satisfy  conditions of $X_{q}(A_1)$-module algebra on $\Q$.

To see that the $X_{q}(A_1)$-module algebra structures are pairwise nonisomorphic,
observe that all the automorphisms of $\Q$ commute with the
actions of $K_{1}$ and $K_2$.
\end{proof}

\section{When $t\neq0$, classification of  $X_{q}(A_{1})$-module algebra structures  on $\Q$   }\label{sect-3}

In this section, we suppose  the automorphism $\Psi$ of $\Q$ as follows:
$$\Psi(x)=\alpha x,~~~\Psi(y)=\beta y+txz,~~~\Psi(z)=\gamma z,~~~(\alpha,\beta,\gamma,t\in \mathbb{C}^{*}),$$
and $\mathrm{Aut}(\Q)\cong \mathbb{C}\rtimes \left(\mathbb{C}^{*}\right)^{3}$.
One can have
 $$\Psi^{-1}(x)=\alpha^{-1} x,~~~\Psi^{-1}(y)=\beta^{-1} y-t\alpha^{-1}\beta^{-1}\gamma^{-1}xz,~~~\Psi^{-1}(z)=\gamma^{-1} z.$$

In the following, we will begin to discuss the $X_{q}(A_{1})$-module algebra structures  on  $\Q$ with $t\neq0$, ie. here $K_{1},K_{2}\in \mathbb{C}\rtimes \left(\mathbb{C}^{*}\right)^{3}$.
In this Section, our research method is similar to Section  \ref{sect-2}.

\subsection{Properties of  $X_{q}(A_{1})$-module algebras  on $\Q$  }\label{sect-3.1}

It is easy to see that any action of $\X$ on $\Q$ is determined by the following $4\times 3$ matrix with entries from $\Q$:
\begin{eqnarray}\label{eqn3-1}
M=\left(
    \begin{array}{cccccc}
      K_{1}(x) & K_{1}(y) & K_{1}(z)\\
      K_{2}(x) & K_{2}(y) & K_{2}(z)\\
      E(x) & E(y) & E(z)\\
      F(x) & F(y) & F(z)
    \end{array}
  \right).
\end{eqnarray}
Given a $\X$-module algebra structure on $\Q$, obviously, the action of $K_{1}$(or $K_{2}$) is determined by an automorphism of $\Q$,
in other words, the actions  of $K_{1}$ and $K_{2}$ are determined by a matrix $M_{K_{1}K_{2}}$ as follows
\begin{eqnarray}\label{eqn3-2}
\begin{array}{lllllll}
M_{K_{1}K_{2}}\overset{definition}=&\left(
    \begin{array}{cccccc}
    K_{1}(x) & K_{1}(y) & K_{1}(z)\\
      K_{2}(x) & K_{2}(y) & K_{2}(z)
    \end{array}
  \right)&~&~&~\\
&=\left(
    \begin{array}{cccccc}
    \alpha_{1}(x) & \beta_{1}(y)+t_{1}xz & \gamma_{1}(z)\\
    \alpha_{2}(x) & \beta_{2}(y)+t_{2}xz & \gamma_{2}(z)\\
    \end{array}
  \right),
\end{array}
\end{eqnarray}
where $\alpha_{i},\beta_{i}, \gamma_{i}, t_i\in \mathbb{C}^{*}$ for $i\in\{1,2\}$.

\begin{lem}\label{lem3-1}
For all $\alpha_{i},\beta_{i}, \gamma_{i}, t_i\in \mathbb{C}^{*}$,  $i\in\{1,2\}$,  either $\beta_{i}=\alpha_{i}\gamma_{i}$ or $t_i=\left(\beta_{i}-\alpha_{i}\gamma_{i}\right)t$, where $t\in \mathbb{C}^{*}.$
\end{lem}
\begin{proof}For all $\alpha_{i},\beta_{i}, \gamma_{i}, t_i\in \mathbb{C}^{*}$,  $i\in\{1,2\}$, we have 
$$K_i(y)=\beta_{i}y+t_ixz~~\mathrm{and}~~K^{-1}_{i}(y)=\beta_{i}^{-1}y-t_i\alpha_{i}^{-1}\beta_{i}^{-1}\gamma_{i}^{-1}xz$$
by (\ref{eqn3-2}).
It is to easy check $K_iK^{-1}_{i}(y)=y,$ and 
\begin{eqnarray*}
K_1K_2(y)&=&K_1(\beta_{2}y+t_2xz)\\
&=&\beta_{1}\beta_{2}y+\beta_{2}t_1xz+t_2\alpha_{1}\gamma_{1}xz,\\
K_2K_1(y)&=&K_2(\beta_{1}y+t_1xz)\\
&=&\beta_{1}\beta_{2}y+\beta_{1}t_2xz+t_1\alpha_{2}\gamma_{2}xz.
\end{eqnarray*}
By the definition of module algebra and (\ref{eqn1-1}), we have $t_1(\beta_{2}-\alpha_{2}\gamma_{2})=t_2(\beta_{1}-\alpha_{1}\gamma_{1})$,
for $t_i\in \mathbb{C}^{*},i=1,2$, hence, either $\beta_{i}=\alpha_{i}\gamma_{i}$ or $\frac{t_1}{t_2}=\frac{\beta_{1}-\alpha_{1}\gamma_{1}}{\beta_{2}-\alpha_{2}\gamma_{2}}$, we can write the latter as $t_i=\left(\beta_{i}-\alpha_{i}\gamma_{i}\right)t$, where $t\in \mathbb{C}^{*}.$
\end{proof}

It is easy to see that every monomial $x^{m_{1}}z^{m_{3}}\in\Q$ is an eigenvector of $K_{1}$(or $K_{2}$),
and the associated eigenvalue $\alpha_{1}^{m_{1}}\gamma_{1}^{m_{3}}$ (or $\alpha_{2}^{m_{1}}\gamma_{2}^{m_{3}}$) is  called the $K_{1}$-weight (or $K_{2}$-weight) of this monomial,
which will be written as
$$wt_{K_{1}}(x^{m_{1}}z^{m_{3}})=\alpha_{1}^{m_{1}}\gamma_{1}^{m_{3}},$$
$$wt_{K_{2}}(x^{m_{1}}z^{m_{3}})=\alpha_{2}^{m_{1}}\gamma_{2}^{m_{3}}.$$

We will also need another matrix  $M_{EF}$ as follows
\begin{eqnarray}\label{eqn2-3}
M_{EF}\overset{definition}=\left(
    \begin{array}{cccccc}
      E(x) & E(y) & E(z)\\
      F(x) & F(y) & F(z)
    \end{array}
  \right).
\end{eqnarray}

Obviously, $K_{1}(x), K_{2}(x), E(x), F(x)$ and $K_{1}(z), K_{2}(z), E(z), F(z)$ are weight vectors for $K_{1}$ and $K_{2}$,  then
\begin{eqnarray}\label{eqn3-4}
\begin{array}{lllllll}
wt_{K_{i}}\left(M\right)&\overset{definition}=\left(
    \begin{array}{cccccc}
    wt_{K_{i}}(K_{1}(x))  & wt_{K_{i}}(K_{1}(z))\\
      wt_{K_{i}}(K_{2}(x))  & wt_{K_{i}}(K_{2}(z))\\
      wt_{K_{i}}(E(x))  & wt_{K_{i}}(E(z))\\
      wt_{K_{i}}(F(x))  & wt_{K_{i}}(F(z))
    \end{array}
  \right)&~&~&~\\
~&\bowtie \left(
    \begin{array}{cccccc}
    wt_{K_{i}}(x)  & wt_{K_{i}}(z)\\
      wt_{K_{i}}(x)   & wt_{K_{i}}(z)\\
     (-1)^{i-1}q^{-1}wt_{K_{i}}(x)   & (-1)^{i-1}q^{-1}wt_{K_{i}}(z)\\
     (-1)^{i-1}qwt_{K_{i}}(x)   & (-1)^{i-1}qwt_{K_{i}}(z)
    \end{array}
  \right)&~&~&~\\
~&= \left(
    \begin{array}{cccccc}
    \alpha_{i}  & \gamma_{i}\\
    \alpha_{i}  & \gamma_{i}\\
     (-1)^{i-1}q^{-1}\alpha_{i}   & (-1)^{i-1}q^{-1}\gamma_{i}\\
      (-1)^{i-1}q\alpha_{i}   & (-1)^{i-1}q\gamma_{i}
    \end{array}
  \right).
\end{array}
\end{eqnarray}

Same as Section  \ref{sect-2},  we denote by $\left(M\right)_{j}$ the $j$-th homogeneous component of $M$.
Obviously, if $\left(M\right)_{j}$ is  nonzero, one can calculate  the associated eigenvalues.

Set $a_{0}, b_{0},c_{0},a'_{0},  b'_{0}, c'_{0}\in \C$, we obtain the $0$-th homogeneous component of $M$ as follow:
\begin{eqnarray}\label{eqn3-5}
\left(M\right)_{0}=\left(
    \begin{array}{cccccc}
      0 &  0 & 0\\
      0 &  0 & 0\\
      a_{0} &  b_{0} & c_{0}\\
     a'_{0} &  b'_{0} & c'_{0}
    \end{array}
  \right)_{0}.
\end{eqnarray}
Then,  we have
\begin{eqnarray}\label{eqn3-6}
\begin{array}{lllllll}
 wt_{K_{i}}\left(\left(M_{EF}\right)_{0}\right)&=\left(
    \begin{array}{cccccc}
     (-1)^{i-1}q^{-1}\alpha_{i}  & (-1)^{i-1}q^{-1}\gamma_{i}\\
      (-1)^{i-1}q\alpha_{i}  & (-1)^{i-1}q\gamma_{i}
    \end{array}
  \right)_{0}
\bowtie \left(
    \begin{array}{cccccc}
   1 &   1\\
   1 &    1
    \end{array}
  \right)_{0},
\end{array}
\end{eqnarray}

According to $q$ is not a root of the unit and relation (\ref{eqn3-6}), it means that $a_{0}$ and $a'_{0}$ ($c_{0}$ and $c'_{0}$) should contain at least one $0$.

An application of $E$ and $F$ to the relations (\ref{eqn1-11})-(\ref{eqn1-13}) by using equation (\ref{eqn3-2}), one has
\begin{eqnarray}\label{eqn3-7}
E(y)x-qE(x)y+K_{2}K_{1}^{-1}(y)E(x)-qK_{2}K_{1}^{-1}(x)E(y)=0, \\ \label{eqn3-8}
E(z)y-qE(y)z+K_{2}K_{1}^{-1}(z)E(y)-qK_{2}K_{1}^{-1}(y)E(z)=0, \\ \label{eqn3-9}
E(z)x-qE(x)z+K_{2}K_{1}^{-1}(z)E(x)-qK_{2}K_{1}^{-1}(x)E(z)=0, \\ \label{eqn3-10}
yF(x)-qxF(y)+F(y)K_{2}^{-1}K_{1}(x)-qF(x)K_{2}^{-1}K_{1}(y)=0, \\\label{eqn3-11}
zF(y)-qyF(z)+F(z)K_{2}^{-1}K_{1}(y)-qF(y)K_{2}^{-1}K_{1}(z)=0, \\\label{eqn3-12}
zF(x)-qxF(z)+F(z)K_{2}^{-1}K_{1}(x)-qF(x)K_{2}^{-1}K_{1}(z)=0. 
\end{eqnarray}
Which certainly implies
\begin{eqnarray*}
b_{0}\left(1-q\alpha_{1}^{-1}\alpha_{2}\right)=a_{0}\left(\beta_{1}^{-1}\beta_{2}-q\right)=
c_{0}\left(1-q\beta_{1}^{-1}\beta_{2}\right)=b_{0}\left(\gamma_{1}^{-1}\gamma_{2}-q\right)=0,\\
c_{0}\left(1-q\alpha_{1}^{-1}\alpha_{2}\right)=a_{0}\left(\gamma_{1}^{-1}\gamma_{2}-q\right)=
a'_{0}\left(1-q\beta_{1}\beta_{2}^{-1}\right)=b'_{0}\left(\alpha_{1}\alpha_{2}^{-1}-q\right)=0,\\
b'_{0}\left(1-q\gamma_{1}\gamma_{2}^{-1}\right)=c'_{0}\left(\beta_{1}\beta_{2}^{-1}-q\right)=
a'_{0}\left(1-q\gamma_{1}\gamma_{2}^{-1}\right)=c'_{0}\left(\alpha_{1}\alpha_{2}^{-1}-q\right)=0.
\end{eqnarray*}

Because $q$ is not a root of the unit, $q\neq \pm1$, and from the above discussion,  
 for the $0$-st homogeneous component $\left(M_{EF}\right)_{0}$ of  $M_{EF}$, we have following lemma.

\begin{lem}\label{lem3-2}
There are $8$ cases for the $0$-st homogeneous component $\left(M_{EF}\right)_{0}$ of  $M_{EF}$, as follows:
\begin{eqnarray}\label{eqn3-13}
&&\left(
    \begin{array}{ccccccccc}
      a_{0} &  0 & 0\\
      0 &  0 & 0
    \end{array}
  \right)_{1} \Rightarrow\alpha_{1}=q,    \alpha_{2}=-q, \beta_{1}^{-1}\beta_{2}=q, \gamma_{1}^{-1}\gamma_{2}=q;\\
&&\left(
    \begin{array}{ccccccccc}
       0 &  b_{0}& 0\\
      0 &  0& 0
    \end{array}
  \right)_{1} \Rightarrow     \alpha_{1}^{-1}\alpha_{2}=q^{-1}, \gamma_{1}^{-1}\gamma_{2}=q;\\
&&\left(
    \begin{array}{ccccccccc}
      0 &  0 & c_{0}\\
      0 &  0 & 0
    \end{array}
  \right)_{1} \Rightarrow\gamma_{1}=q,    \gamma_{2}=-q, \beta_{1}^{-1}\beta_{2}=q^{-1}, \alpha_{1}^{-1}\alpha_{2}=q^{-1};\\
&&\left(
    \begin{array}{ccccccccc}
       0 & 0& 0\\
       a'_{0} &  0& 0
    \end{array}
  \right)_{1} \Rightarrow   \alpha_{1}=q^{-1},\alpha_{2}=-q^{-1}, \beta_{1}\beta^{-1}_{2}=q^{-1}, \gamma_{1}\gamma^{-1}_{2}=q^{-1};\\
&&\left(
    \begin{array}{ccccccccc}
      0 &  0 & 0\\
      0 &  b'_{0} & 0
    \end{array}
  \right)_{1} \Rightarrow    \alpha_{1}\alpha^{-1}_{2}=q, \gamma_{1}\gamma_{2}^{-1}=q^{-1};\\
&&\left(
    \begin{array}{ccccccccc}
       0 & 0& 0\\
      0 &  0&  c'_{0}
    \end{array}
  \right)_{1} \Rightarrow  \gamma_{1}=q^{-1}, \gamma_{2}=-q^{-1}    \beta_{1}\beta^{-1}_{2}=q, \alpha_{1}\alpha^{-1}_{2}=q;\\
&&\left(
    \begin{array}{ccccccccc}
       0&  b_{0} & 0\\
      0 &  b'_{0}  & 0
    \end{array}
  \right)_{1} \Rightarrow \alpha_{1}^{-1}\alpha_{2}=q^{-1}, \gamma_{1}^{-1}\gamma_{2}=q;\\
&&\left(
    \begin{array}{ccccccccc}
       0 &  0& 0\\
      0 &  0& 0
    \end{array}
  \right)_{1} \Rightarrow  \mathrm{it~ does~ not~ determine~ the~ weight~ constants ~at~ all.} 
\end{eqnarray}
\end{lem}

Next, for the $1$-st homogeneous component, due to $q$ is not a root of the unit, one has
\begin{eqnarray*}
&&wt_{K_{1}}(E(x)) = q^{-1}\alpha_{1}=q^{-1}wt_{K_{1}}(x)\neq wt_{K_{1}}(x),\\
&&wt_{K_{2}}(E(x)) =-q^{-1}\alpha_{2}=-q^{-1}wt_{K_{2}}(x)\neq wt_{K_{2}}(x),
\end{eqnarray*}
which implies
$$(E(x))_{1}=a_{1}z,$$
for some  $ a_{1}\in \C$,
and in a similar way we have
{\small$$\left(M_{EF}\right)_{1}=\left(\begin{array}{cccccc}
a_{1}z   &  b_{1}x+b_{2}y+b_{3}z    &c_{1}x\\
a'_{1}z &  b'_{1}x+b'_{2}y+b'_{3}z  &c'_{1}x
    \end{array}
  \right)_{1}$$}
where $b_{1},b_{2},b_{3},c_{1},a'_{1},b'_{1},b'_{2},b'_{3},c'_{1}\in \C$.
In fact,
\begin{eqnarray}
\begin{array}{lllllll}
wt_{K_{i}}\left((M_{EF})_{1}\right)&\bowtie \left(
    \begin{array}{cccccc}
   (-1)^{i-1}q^{-1}\alpha_{i}  & (-1)^{i-1}q^{-1}\gamma_{i}\\
      (-1)^{i-1}q\alpha_{i}  & (-1)^{i-1}q\gamma_{i}
    \end{array}
  \right)\bowtie \left(
    \begin{array}{cccccc}
       \gamma_{i} & \alpha_{i}\\
     \gamma_{i} & \alpha_{i}
    \end{array}
  \right).
\end{array}
\end{eqnarray}\label{eqn2-4}

Now, we project (\ref{eqn3-7})-(\ref{eqn3-12}) to $\Q_{2}$, one can  obtain the following conclusion.

If $a_{0}\neq0$, then $\alpha_{1}=q,    \alpha_{2}=-q, \beta_{1}^{-1}\beta_{2}=q, \gamma_{1}^{-1}\gamma_{2}=q,$ and we have
$$(E(y))_{1}x-q(E(x))_{1}y+qy(E(x))_{1}+(t_{2}\beta^{-1}_{1}+qt_{1}\beta^{-1}_{1})xz(E(x))_{0}+qx(E(y))_{1}=0,$$
$$b_{1}\left(1+q\right)x^{2}+2qb_{2}xy+a_{1}q\left(1-q\right)yz
+(2qb_{3}+a_{0}t_{2}\beta^{-1}_{1}+qa_{0}t_{1}\beta^{-1}_{1})xz=0,$$
\begin{eqnarray}
\begin{cases}
a_{1}=0,\\
 b_{1}=0,\\
b_{2}=0,\\
b_{3}=\frac{-a_{0}(t_{2}+qt_{1})}{2q\beta_{1}}.
\end{cases}
\end{eqnarray}

If $a_{0}=0$, then
$$(E(y))_{1}x-q(E(x))_{1}y+\beta_{1}^{-1}\beta_{2}y(E(x))_{1}-q\alpha_{1}^{-1}\alpha_{2}x(E(y))_{1}=0,$$
$$b_{1}\left(1-q\alpha_{1}^{-1}\alpha_{2}\right)x^{2}+b_{2}\left(1-\alpha_{1}^{-1}\alpha_{2}\right)yx+b_{3}\left(1-\alpha_{1}^{-1}\alpha_{2}\right)zx
+a_{1}\left(\beta_{1}^{-1}\beta_{2}-q^{2}\right)yz=0,$$
\begin{eqnarray}
\begin{cases}
b_{1}\neq0\Rightarrow \alpha_{1}^{-1}\alpha_{2}=q^{-1},\\
b_{2}\neq0\Rightarrow \alpha_{1}^{-1}\alpha_{2}=1,\\
b_{3}\neq0\Rightarrow \alpha_{1}^{-1}\alpha_{2}=1,\\
a_{1}\neq0\Rightarrow\beta_{1}^{-1}\beta_{2}=q^{2}.
\end{cases}
\end{eqnarray}

If $c_{0}\neq0$, then $\gamma_{1}=q,    \gamma_{2}=-q, \beta_{1}^{-1}\beta_{2}=q^{-1}, \alpha_{1}^{-1}\alpha_{2}=q^{-1},$ and we have
$$(E(z))_{1}y-q(E(y))_{1}z-z(E(y))_{1}-y(E(z))_{1}-q(t_{2}\beta^{-1}_{1}+q^{-1}t_{1}\beta^{-1}_{1})xz(E(z))_{0}=0,$$
$$c_{1}\left(1-q\right)xy-2qb_{2}yz-b_{3}\left(1+q\right)z^{2}
-(2qb_{1}+qc_{0}t_{2}\beta^{-1}_{1}+c_{0}t_{1}\beta^{-1}_{1})xz=0,$$
\begin{eqnarray}
\begin{cases}
b_{2}=0,\\
b_{3}=0,\\
c_{1}=0,\\
b_{1}=\frac{-c_{0}(t_{1}+qt_{2})}{2q\beta_{1}}.
\end{cases}
\end{eqnarray}

If $c_{0}=0$, then
$$(E(z))_{1}y-q(E(y))_{1}z+\gamma_{1}^{-1}\gamma_{2}z(E(y))_{1}-q\beta_{1}^{-1}\beta_{2}y(E(z))_{1}=0,$$
$$c_{1}\left(1-q^{2}\beta_{1}^{-1}\beta_{2}\right)xy+qb_{1}\left(\gamma_{1}^{-1}\gamma_{2}-1\right)xz+qb_{2}\left(\gamma_{1}^{-1}\gamma_{2}-1\right)yz
+b_{3}\left(\gamma_{1}^{-1}\gamma_{2}-q\right)z^{2}=0,$$
\begin{eqnarray}
\begin{cases}
c_{1}\neq0\Rightarrow \beta_{1}^{-1}\beta_{2}=q^{-2},\\
b_{1}\neq0\Rightarrow \gamma_{1}^{-1}\gamma_{2}=1,\\
b_{2}\neq0\Rightarrow \gamma_{1}^{-1}\gamma_{2}=1,\\
b_{3}\neq0\Rightarrow\gamma_{1}^{-1}\gamma_{2}=q.
\end{cases}
\end{eqnarray}

$$(E(z))_{1}x-q(E(x))_{1}z+\gamma_{1}^{-1}\gamma_{2}z(E(x))_{1}-q\alpha_{1}^{-1}\alpha_{2}x(E(z))_{1}=0,$$
$$c_{1}(1-q\alpha_{1}^{-1}\alpha_{2})x^{2}+a_{1}(\gamma_{1}^{-1}\gamma_{2}-q)z^{2}=0,$$
\begin{eqnarray}
\begin{cases}
c_{1}\neq0\Rightarrow \alpha_{1}^{-1}\alpha_{2}=q^{-1},\\
a_{1}\neq0\Rightarrow \gamma_{1}^{-1}\gamma_{2}=q.
\end{cases}
\end{eqnarray}

If $a'_{0}\neq0$, then $\alpha_{1}=q^{-1},    \alpha_{2}=-q^{-1}, \beta_{1}\beta_{2}^{-1}=q^{-1}, \gamma_{1}\gamma_{2}^{-1}=q^{-1},$ and we have
$$y(F(x))_{1}-qx(F(y))_{1}-(F(y))_{1}x-(F(x))_{1}y-q(F(x))_{0}(t_{1}\alpha^{-1}_{2}\gamma^{-1}_{2}-q^{-1}t_{2}\alpha^{-1}_{2}\gamma^{-1}_{2})xz=0,$$
$$a'_{1}\left(1-q\right)yz-b'_{1}\left(1+q\right)x^{2}-2qb'_{2}xy
+(-2qb'_{3}+a'_{0}t_{2}\alpha^{-1}_{2}\gamma^{-1}_{2}-qa'_{0}t_{1}\alpha^{-1}_{2}\gamma^{-1}_{2})xz=0,$$
\begin{eqnarray}
\begin{cases}
a'_{1}=0,\\
b'_{1}=0,\\
b'_{2}=0,\\
b'_{3}=\frac{a'_{0}(t_{2}-qt_{1})}{2q\alpha_{2}\gamma_{2}}.
\end{cases}
\end{eqnarray}

If $a'_{0}=0$, then
$$y(F(x))_{1}-qx(F(y))_{1}+\alpha_{1}\alpha_{2}^{-1}(F(y))_{1}x-q\beta_{1}\beta_{2}^{-1}(F(x))_{1}y=0,$$
$$b'_{1}\left(\alpha_{1}\alpha_{2}^{-1}-q\right)x^{2}+qb'_{2}\left(\alpha_{1}\alpha_{2}^{-1}-1\right)xy
+qb_{3}\left(\alpha_{1}\alpha_{2}^{-1}-1\right)xz+a'_{1}\left(1-q^{2}\beta_{1}\beta_{2}^{-1}\right)yz=0,$$
\begin{eqnarray}
\begin{cases}
a'_{1}\neq0\Rightarrow\beta_{1}\beta_{2}^{-1}=q^{-2},\\
b'_{1}\neq0\Rightarrow \alpha_{1}\alpha_{2}^{-1}=q,\\
b'_{2}\neq0\Rightarrow \alpha_{1}\alpha_{2}^{-1}=1,\\
b'_{3}\neq0\Rightarrow \alpha_{1}\alpha_{2}^{-1}=1.
\end{cases}
\end{eqnarray}

If $c'_{0}\neq0$, then $\gamma_{1}=q^{-1},   \gamma_{2}=-q^{-1}, \beta_{1}\beta_{2}^{-1}=q,\alpha_{1}\alpha_{2}^{-1}=q,$ and we have
$$z(F(y))_{1}-qy(F(z))_{1}+q(F(z))_{1}y
+(F(z))_{0}(t_{1}\alpha^{-1}_{2}\gamma^{-1}_{2}-qt_{2}\alpha^{-1}_{2}\gamma^{-1}_{2})xz+q(F(y))_{1}z=0,$$
$$2qb'_{2}yz+b'_{3}\left(1+q\right)z^{2}+qc'_{1}(1-q)xy
+(2qb'_{1}-qc'_{0}t_{2}\alpha^{-1}_{2}\gamma^{-1}_{2}+c'_{0}t_{1}\alpha^{-1}_{2}\gamma^{-1}_{2})xz=0,$$
\begin{eqnarray}
\begin{cases}
b'_{2}=0,\\
b'_{3}=0,\\
c'_{1}=0,\\
b'_{1}=\frac{c'_{0}(qt_{2}-t_{1})}{2q\alpha_{2}\gamma_{2}}.
\end{cases}
\end{eqnarray}

If $c'_{0}=0$, then
$$z(F(y))_{1}-qy(F(z))_{1}+\beta_{1}\beta_{2}^{-1}(F(z))_{1}y-q\gamma_{1}\gamma^{-1}_{2}(F(y))_{1}z=0,$$
$$qb'_{1}\left(1-\gamma_{1}\gamma_{2}^{-1}\right)xz+qb'_{2}\left(1-\gamma_{1}\gamma_{2}^{-1}\right)yz
+b'_{3}\left(1-q\gamma_{1}\gamma_{2}^{-1}\right)z^{2}+c'_{1}\left(\beta_{1}\beta_{2}^{-1}-q^{2}\right)xy=0,$$
\begin{eqnarray}
\begin{cases}
b'_{1}\neq0\Rightarrow \gamma_{1}\gamma_{2}^{-1}=1,\\
b'_{2}\neq0\Rightarrow \gamma_{1}\gamma_{2}^{-1}=1,\\
b'_{3}\neq0\Rightarrow \gamma_{1}\gamma_{2}^{-1}=q^{-1},\\
c'_{1}\neq0\Rightarrow\beta_{1}\beta_{2}^{-1}=q^{2}.
\end{cases}
\end{eqnarray}

$$z(F(x))_{1}x-qx(F(z))_{1}+\alpha_{1}\alpha_{2}^{-1}(F(z))_{1}x-q\gamma_{1}\gamma_{2}^{-1}(F(x))_{1}z=0,$$
$$c'_{1}(\alpha_{1}\alpha_{2}^{-1}-q)x^{2}+a'_{1}(1-q\gamma_{1}\gamma_{2}^{-1})z^{2}=0,$$
\begin{eqnarray}
\begin{cases}
a'_{1}\neq0\Rightarrow \gamma_{1}\gamma_{2}^{-1}=q^{-1},\\
c'_{1}\neq0\Rightarrow \alpha_{1}\alpha_{2}^{-1}=q.
\end{cases}
\end{eqnarray}

From the above discussion, and due to $q$ is not a root of the unit, we can obtain the following lemma.
\begin{lem}\label{lem3-3}
There are $18$ cases for the $1$-st homogeneous component $\left(M_{EF}\right)_{1}$ of  $M_{EF}$, as follows:
\begin{eqnarray}
&&\left(
    \begin{array}{ccccccccc}
      a_{1}y & 0&0\\
      0      & 0&0
    \end{array}
  \right)_{1} \Rightarrow \beta_{1}=q^{-1}\alpha_{1}, \beta_{2}=-q^{-1}\alpha_{2},\beta_{1}^{-1}\beta_{2}=q, \gamma_{1}^{-1}\gamma_{2}=1;\\
&&\left(
    \begin{array}{ccccccccc}
      a_{2}z & 0&0\\
      0      & 0&0
    \end{array}
  \right)_{1} \Rightarrow \gamma_{1}=q^{-1}\alpha_{1}, \gamma_{2}=-q^{-1}\alpha_{2},\beta_{1}^{-1}\beta_{2}=q^{2}, \gamma_{1}^{-1}\gamma_{2}=q;\\
&&\left(
    \begin{array}{ccccccccc}
      0 &b_{1}x&0\\
      0      & 0&0
    \end{array}
  \right)_{1} \Rightarrow \beta_{1}=q\alpha_{1}, \beta_{2}=-q\alpha_{2},\alpha_{1}^{-1}\alpha_{2}=q^{-1}, \gamma_{1}^{-1}\gamma_{2}=1;\\
&&\left(
    \begin{array}{ccccccccc}
      0 &b_{2}z&0\\
      0      & 0&0
    \end{array}
  \right)_{1}\Rightarrow \gamma_{1}=q^{-1}\beta_{1}, \gamma_{2}=-q^{-1}\beta_{2},\alpha_{1}^{-1}\alpha_{2}=1, \gamma_{1}^{-1}\gamma_{2}=q;\\
&&\left(
    \begin{array}{ccccccccc}
      0 &0&c_{1}x\\
      0 & 0&0
    \end{array}
  \right)_{1} \Rightarrow \alpha_{1}=q^{-1}\gamma_{1}, \alpha_{2}=-q^{-1}\gamma_{2},\alpha_{1}^{-1}\alpha_{2}=q^{-1}, \beta_{1}^{-1}\beta_{2}=q^{-2};\\
&&\left(
    \begin{array}{ccccccccc}
      0 &0&c_{2}y\\
      0      & 0&0
    \end{array}
  \right)_{1}\Rightarrow \beta_{1}=q^{-1}\gamma_{1}, \beta_{2}=-q^{-1}\gamma_{2},\alpha_{1}^{-1}\alpha_{2}=1, \beta_{1}^{-1}\beta_{2}=q^{-1};\\
&&\left(
    \begin{array}{ccccccccc}
       0& 0&0\\
      a'_{1}y      & 0&0
    \end{array}
  \right)_{1} \Rightarrow \beta_{1}=q\alpha_{1}, \beta_{2}=-q\alpha_{2},\beta_{1}^{-1}\beta_{2}=q, \gamma_{1}^{-1}\gamma_{2}=1;\\
&&\left(
    \begin{array}{ccccccccc}
      0 & 0&0\\
      a'_{2}z & 0&0
    \end{array}
  \right)_{1} \Rightarrow \gamma_{1}=q\alpha_{1}, \gamma_{2}=-q\alpha_{2},\beta_{1}^{-1}\beta_{2}=q^{2}, \gamma_{1}^{-1}\gamma_{2}=q;
\end{eqnarray}
\begin{eqnarray}
&&\left(
    \begin{array}{ccccccccc}
      0 &0&0\\
      0 &b'_{1}x&0
    \end{array}
  \right)_{1}\Rightarrow \beta_{1}=q^{-1}\alpha_{1}, \beta_{2}=-q^{-1}\alpha_{2},\alpha_{1}^{-1}\alpha_{2}=q^{-1}, \gamma_{1}^{-1}\gamma_{2}=1;\\
&&\left(
    \begin{array}{ccccccccc}
      0 &0&0\\
      0& b'_{2}z&0
    \end{array}
  \right)_{1} \Rightarrow \gamma_{1}=q\beta_{1}, \gamma_{2}=-q\beta_{2},\alpha_{1}^{-1}\alpha_{2}=1, \gamma_{1}^{-1}\gamma_{2}=q;\\
&&\left(
    \begin{array}{ccccccccc}
      0 &0&0\\
      0 & 0&c'_{1}x
    \end{array}
  \right)_{1} \Rightarrow \alpha_{1}=q\gamma_{1}, \alpha_{2}=-q\gamma_{2},\alpha_{1}^{-1}\alpha_{2}=q^{-1}, \beta_{1}^{-1}\beta_{2}=q^{-2};\\
&&\left(
    \begin{array}{ccccccccc}
      0 &0&0\\
      0 &0&c'_{2}y
    \end{array}
  \right)_{1} \Rightarrow \beta_{1}=q\gamma_{1}, \beta_{2}=-q\gamma_{2},\alpha_{1}^{-1}\alpha_{2}=1, \beta_{1}^{-1}\beta_{2}=q^{-1};\\
&&\left(
    \begin{array}{ccccccccc}
      0 &0&0\\
      0 &0&0
    \end{array}
  \right)_{1} it~ does~ not~ determine~ the~ weight~ constants ~at~ all;
\end{eqnarray}
If $a_{0}\neq0$, then $\alpha_{1}=q,    \alpha_{2}=-q, \beta_{1}^{-1}\beta_{2}=q, \gamma_{1}^{-1}\gamma_{2}=q,$ and we have
\begin{eqnarray}
\begin{cases}
a_{1}=0,\\
 b_{1}=0,\\
b_{2}=0,\\
b_{3}=\frac{-a_{0}(t_{2}+qt_{1})}{2q\beta_{1}}.
\end{cases}
\end{eqnarray}
 If $c_{0}\neq0$, then $\gamma_{1}=q, \gamma_{2}=-q, \beta_{1}^{-1}\beta_{2}=q^{-1}, \alpha_{1}^{-1}\alpha_{2}=q^{-1},$ and we have
\begin{eqnarray}
\begin{cases}
b_{2}=0,\\
b_{3}=0,\\
c_{1}=0,\\
b_{1}=\frac{-c_{0}(t_{1}+qt_{2})}{2q\beta_{1}}.
\end{cases}
\end{eqnarray}
If $a'_{0}\neq0$, then $\alpha_{1}=q^{-1},    \alpha_{2}=-q^{-1}, \beta_{1}\beta_{2}^{-1}=q^{-1}, \gamma_{1}\gamma_{2}^{-1}=q^{-1},$ and we have
\begin{eqnarray}
\begin{cases}
a'_{1}=0,\\
b'_{1}=0,\\
b'_{2}=0,\\
b'_{3}=\frac{a'_{0}(t_{2}-qt_{1})}{2q\alpha_{2}\gamma_{2}}.
\end{cases}
\end{eqnarray}
If $c'_{0}\neq0$, then $\gamma_{1}=q^{-1},   \gamma_{2}=-q^{-1}, \beta_{1}\beta_{2}^{-1}=q,\alpha_{1}\alpha_{2}^{-1}=q,$ and we have
\begin{eqnarray}
\begin{cases}
b'_{2}=0,\\
b'_{3}=0,\\
c'_{1}=0,\\
b'_{1}=\frac{c'_{0}(qt_{2}-t_{1})}{2q\alpha_{2}\gamma_{2}}.
\end{cases}
\end{eqnarray}
\end{lem}

\subsection{The structures of  $X_{q}(A_{1})$-module algebra  on $\Q$  }\label{sect-3}
Through the previous discussion, 
we found that both the $0$-st homogeneous component $\left(M_{EF}\right)_{0}$ 
and the $1$-st homogeneous component $\left(M_{EF}\right)_{1}$ determine the eigenvalues of  $x$ and $z$. 
By lemma \ref{lem3-2} and lemma \ref{lem3-3}, and $q$ is not a root of the unit, 
it follows that there are 91 kinds of $\left[\left(M_{EF}\right)_{0},\left(M_{EF}\right)_{1}\right]$ are empty.
Hence, we only discuss the following cases.

\begin{lem}\label{lem3-4} If  the $0$-th homogeneous component of $M_{EF}$ is zero and the $1$-st homogeneous component of $M_{EF}$ is nonzero, then these  series are empty. 
\end{lem}
\begin{proof} The  proof is similar to the proof of Lemma \ref{lem2-3}.
\end{proof}

\begin{lem}\label{lem3-5} If  the $0$-th homogeneous component of $M_{EF}$ is nonzero and the $1$-st homogeneous component of $M_{EF}$ is zero, then these  series are empty. 
\end{lem}
\begin{proof}The  proof is similar to the proof of Lemma \ref{lem2-4}.
\end{proof}

\begin{lem}\label{lem3-6} The  {\small$\left[\left(
    \begin{array}{ccccccccc}
      a_0 & 0&0\\
      0     &0&0
    \end{array}
  \right)_{0},\left(
    \begin{array}{ccccccccc}
      0 & \frac{-a_0(t_2+qt_1)}{2q\beta_1}z&0\\
      0 & 0&0
    \end{array}
  \right)_{1}\right]$}
-series is empty.
\end{lem}
\begin{proof}By (\ref{eqn3-13}), one has $\alpha_{1}=q,    \alpha_{2}=-q, \beta_{1}^{-1}\beta_{2}=q, \gamma_{1}^{-1}\gamma_{2}=q$, 
and 
$$\alpha_{1}^{-1}\alpha_{1}\gamma_{1}^{-1}\gamma_{2}=-q\neq q=\beta_{1}^{-1}\beta_{2}.$$
From Lemma \ref{lem3-1}, it can be concluded that $t_i=\left(\beta_{i}-\alpha_{i}\gamma_{i}\right)t~ (i=1,2)$, where $t\in \mathbb{C}^{*}$, and $\frac{-a_0(t_2+qt_1)}{2q\beta_1}z=-a_0tz$.

If we suppose this series is not empty, we have $K_{1}(x)=qx, K_{2}(x)=-qx,$ and $wt_{K_i}(E(x))=(-1)^{i-1}q^{-1}\alpha_{i}=1$, hence $E(x)=a_0$.
Set 
 \begin{eqnarray*}
\begin{array}{lllllll}
&K_{1}(y)=\beta_{1}y+(\beta_{1}-q\gamma_1)txz,\; & K_{2}(y)=\beta_{2}y+(\beta_{2}+q\gamma_2)txz,\; \\
&K_{1}(z)=\gamma_{1}z,\; & K_{2}(z)=\gamma_{2}z,\; \\
&E(y)=-a_0tz+\sum\limits_{m_{2}+n_{2}+l_{2}\geq2}\psi^{2}_{m_{2}n_{2}l_{2}}x^{m_{2}}y^{n_{2}}z^{l_{2}}\;
&\mathrm{for}~~{m_{2},n_{2},l_{2}}\in \mathbb{N},\\
&E(z)=\sum\limits_{m_{3}+l_{3}\geq2}\psi^{3}_{m_{3}l_{3}}x^{m_{3}}z^{l_{3}}\;
&\mathrm{for}~~{m_{3},l_{3}}\in \mathbb{N},\\
&F(x)=\sum\limits_{m_{4}+l_{4}\geq2}\psi^{4}_{m_{4}l_{4}}x^{m_{4}}z^{l_{4}}\;
&\mathrm{for}~~{m_{4},l_{4}}\in \mathbb{N},\\
&F(y)=\sum\limits_{m_{5}+n_{5}+l_{5}\geq2}\psi^{5}_{m_{5}n_{5}l_{5}}x^{m_{5}}y^{n_{5}}z^{l_{5}}\;
&\mathrm{for}~~{m_{5},n_{5},l_{5}}\in \mathbb{N},\\
&F(z)=\sum\limits_{m_{6}+l_{6}\geq2}\psi^{6}_{m_{6}l_{6}}x^{m_{6}}z^{l_{6}}\;
&\mathrm{for}~~{m_{6},l_{6}}\in \mathbb{N},
\end{array}
\end{eqnarray*}
where  $\beta_{1},\beta_{2},\gamma_{1},\gamma_{2}, t\in \mathbb{C}^{\ast}$, and $\psi^{2}_{m_{2}n_{2}l_{2}}, \psi^{3}_{m_{3}l_{3}},\psi^{4}_{m_{4}l_{4}},\psi^{5}_{m_{5}n_{5}l_{5}},\psi^{6}_{m_{6}l_{6}}\in \C.$
According to (\ref{eqn3-7}), it can be obtained
$$\sum\limits_{m_{2}+n_{2}+l_{2}\geq2}\psi^{2}_{m_{2}n_{2}l_{2}}(q^{n_2+l_2}+q)x^{m_{2}+1}y^{n_{2}}z^{l_{2}}=0,$$
then for all $m_{2},n_{2},l_{2}\in \mathbb{N}$ with $m_{2}+n_{2}+l_{2}\geq2$, one has $\psi^{2}_{m_{2}n_{2}l_{2}}=0$ and 
$E(y)=-a_0tz.$
Similarly, we can get $E(z)=0$.
By (\ref{eqn1-3}) and (\ref{eqn1-5}), we have
{\small\begin{align*}
(K_{1}F-qFK_{1})(x)=&K_{1}(F(x))-qF(K_{1}(x))\\
=&K_{1}(\sum\limits_{m_{4}+l_{4}\geq2}\psi^{4}_{m_{4}l_{4}}x^{m_{4}}z^{l_{4}})-q^{2}F(x)\\
=&\sum\limits_{m_{4}+l_{4}\geq2}\psi^{4}_{m_{4}l_{4}}(q^{m_4}\gamma_{1}^{l_4}-q^2)x^{m_{4}}z^{l_{4}}=0,\\
(K_{2}F+qFK_{2})(x)=&K_{2}(F(x))+qF(K_{2}(x))\\
=&\sum\limits_{m_{4}+l_{4}\geq2}\psi^{4}_{m_{4}l_{4}}[(-q)^{m_4}\gamma_{2}^{l_4}-q^2]x^{m_{4}}z^{l_{4}}=0,
\end{align*}}
then for all $m_{4},l_{4}\in \mathbb{N}$ with $m_{4}+l_{4}\geq2$, one has
$$\psi^{4}_{m_{4}l_{4}}=0~~\mathrm{or} ~~(-q)^{m_{4}}\gamma_{2}^{l_{4}}=q^2,$$
and $F(x)=0$ or $F(x)=\psi^{4}_{20}x^{2}$.
\begin{itemize}
  \item If $F(x)=0$, it is easy to get $F(y)=F(z)=0$, then 
  $$(EF-FE)(z)=0\neq \frac{K_{2}K_{1}^{-1}-K_{2}^{-1}K_{1}}{q-q^{-1}}(z)=z,$$
this contradicts our hypothesis.
  \item If $F(x)=\psi^{4}_{20}x^{2}$, by (\ref{eqn3-12}), one can get $F(z)=\frac{\psi^{4}_{20}(q^2-1)}{2q}xz$,
  and 
   $$(EF-FE)(z)=\frac{K_{2}K_{1}^{-1}-K_{2}^{-1}K_{1}}{q-q^{-1}}(z)=z,$$
after calculation, we can conclude that $F(x)=\frac{2qa_0}{q^2-1}x^2$ and
$F(z)=a_0xz$. However
$$F^2(z)=F(a_0xz)=a_{0}^{2}\frac{q^2+1}{q^2-1}x^2z\neq 0,$$
this contradicts our hypothesis.
\end{itemize}

In summary, this series is empty.
\end{proof}

Similar to Lemma 3.6, we can obtain

 {\small$\left[\left(
    \begin{array}{ccccccccc}
      0 & 0&c_0\\
      0     &0&0
    \end{array}
  \right)_{0},\left(
    \begin{array}{ccccccccc}
      0 & \frac{-c_0(t_1+qt_2)}{2q\beta_1}x&0\\
      0 & 0&0
    \end{array}
  \right)_{1}\right],\left[\left(
    \begin{array}{ccccccccc}
      0& 0&0\\
      a'_0     &0&0
    \end{array}
  \right)_{0},\left(
    \begin{array}{ccccccccc}
      0 & 0&0\\
      0 & \frac{a'_0(t_2-qt_1)}{2q\alpha_2\gamma_2}z&0
    \end{array}
  \right)_{1}\right],$} 
  
  {\small$\left[\left(
    \begin{array}{ccccccccc}
      0 & 0&0\\
      C'_0 &0&0
    \end{array}
  \right)_{0},\left(
    \begin{array}{ccccccccc}
      0 & 0z&0\\
      0 & \frac{c'_0(qt_2-t_1)}{2q\alpha_2\gamma_2}z&0
    \end{array}
  \right)_{1}\right]$} are empty series.
  
Next we turn to "nonempty" series, it only has one "nonempty" series.

\begin{thm}\label{thm3-7} The {\small$\left[\left(
    \begin{array}{ccccccccc}
     0 & 0&0\\
      0 & 0&0
    \end{array}
  \right)_{0},\left(
    \begin{array}{ccccccccc}
      0 & 0&0\\
      0 & 0&0
    \end{array}
  \right)_{1}\right]$}
-series has two types of  $X_{q}(A_1)$-module algebra structures  on the $\Q$ given by
\begin{enumerate}
  \item for all $\lambda, \mu, t\in \C^{*}$, we have
\begin{eqnarray}\label{eqn3-49}
\begin{array}{llll}
K(x)=\lambda x, &K_2(x)=\pm\lambda x\\
K(y)=\lambda\mu y+txz, &K_2(x)=\pm(\lambda\mu y+txz)\\
K(z)=\mu z, &K_2(z)=\pm\mu z\\
E(x)=E(y)=E(z)=0,&F(x)=F(y)=F(z)=0,
\end{array}
\end{eqnarray}
they are pairwise nonisomorphic.  
  \item for all $\lambda,\sigma, \mu, \widetilde{t}\in \C^{*}$, we have
\begin{eqnarray}\label{eqn3-50}
\begin{array}{llll}
K(x)=\lambda x, &K_2(x)=\pm\lambda x\\
K(y)=\sigma y+\widetilde{t}xz, &K_2(x)=\pm(\sigma y+\widetilde{t}xz)\\
K(z)=\mu z, &K_2(z)=\pm\mu z\\
E(x)=E(y)=E(z)=0,&F(x)=F(y)=F(z)=0,
\end{array}
\end{eqnarray}
where $\widetilde{t}=(\lambda\mu-\sigma)t\in \C^{*}$, they are pairwise nonisomorphic.
\end{enumerate}
\end{thm}
\begin{proof}
The  proof is similar to the proof of Theorm \ref{thm2-5}.
\end{proof}

\end{document}